\documentclass[11pt]{article}
\usepackage[centertags]{amsmath}
\usepackage{amsfonts}
\usepackage{amssymb}
\usepackage{amsthm}
\usepackage{newlfont}
\usepackage{stmaryrd}
\usepackage[]{titletoc}  
\usepackage{mathrsfs}
\usepackage{stmaryrd}
\usepackage{graphicx}

\usepackage[centertags]{amsmath}
\usepackage{amsfonts}

\usepackage{hyperref}

\newlength{\defbaselineskip}
\newcommand{\setlinespacing}[1]%
           {\setlength{\baselineskip}{#1 \defbaselineskip}}

\theoremstyle{plain}
\newtheorem{thm}{Theorem}[section]
\newtheorem{cor}[thm]{Corollary}
\newtheorem{lem}[thm]{Lemma}
\newtheorem{prop}[thm]{Proposition}
\newtheorem{exam}[thm]{Example}
\newtheorem{rem}[thm]{Remark}
\newtheorem{Def}[thm]{Definition}

\makeatletter\@addtoreset{equation}{section} \makeatother
\begin{document}
\title {Geometric constructions of thin Blaschke products and reducing subspace problem}
\author{Kunyu Guo\hskip6mm  Hansong Huang}
\date{}
 \maketitle \noindent\textbf{Abstract:}
In this paper, we mainly  study   geometric constructions of thin
Blaschke products $B$ and  reducing subspace problem of
 multiplication operators  induced by such symbols $B$ on the Bergman space.
   Considering  such   multiplication operators  $M_B$,
  we present a representation of those operators  commuting with both $M_B$ and
  $M_B^*$.   It is shown that for ``most" thin Blaschke products $B$, $M_B$ is irreducible, i.e. $M_B$  has no nontrivial reducing
subspace; and such a  thin Blaschke product $B$ is constructed.
 As an application of the methods, it is proved that
 for ``most" finite Blaschke products $\phi$, $M_\phi$ has exactly two minimal reducing subspaces.
 Furthermore, under a mild condition, we get a geometric characterization for when $M_B$ defined by a thin Blaschke product $B$
  has a nontrivial reducing subspace.



\section{Introduction}
~~~~Let $\mathbb{D}$ be the open unit disk in the complex plane
$\mathbb{C}$.
 Denote by $L^2_a(\mathbb{D})$ the Bergman space consisting of
all holomorphic functions over $\mathbb{D}$ which are square
integrable with respect to  $dA$, the normalized area measure on
$\mathbb{D}$.
 For any bounded holomorphic function $\phi$ over $\mathbb{D}$, let $M_\phi$ be the multiplication operator
 defined on  $L^2_a(\mathbb{D})$ by the symbol $\phi$. Recall that, a closed  subspace $M$
 is called  a \emph{reducing subspace} of $M_\phi$  if $M_\phi M\subseteq M$ and $M_\phi^*M\subseteq M$.
 In this paper,  let $\mathcal{W}^*(\phi)$ be the von Neumann
 algebra generated by $M_\phi$ and denote
 $ \mathcal{V}^*(\phi)\triangleq \mathcal{W}^*(\phi)'$, the commutant algebra of $\mathcal{W}^*(\phi)$.
 For each closed subspace $M$, denote by $P_M$ the orthogonal projection onto $M$. Then
 one sees that  there is a one-to-one and onto correspondence: $M\mapsto P_M$, which maps reducing
subspaces of $M_\phi$  to projections in $\mathcal{V}^*(\phi)$.
Since a von Neumann algebra is generated by its projections,
 the study of $\mathcal{V}^*(\phi)$ is reduced to studying the lattice of projections in $\mathcal{V}^*(\phi)$, which
 is equivalent to studying the reducing subspaces for $M_\phi$.
By von Neumann bicommutant theorem $\mathcal{V}^*(\phi)'= \mathcal{W}^*(\phi)$,  which
 help us to understand  the structure of $\mathcal{V}^*(\phi)$ and $ \mathcal{W}^*(\phi)$.

If  $ M_\phi  $ has no nontrivial reducing subspace, then  $ M_\phi$ is called \emph{irreducible}. This means  there is no nontrivial
projection  in $ \mathcal{V}^*(\phi)$, and thus $\mathcal{V}^*(\phi)= \mathbb{C}I$. Therefore, by von Neumann
bicommutant theorem, $\mathcal{W}^*(\phi)=B(L^2_a(\mathbb{D}))$, all
bounded operators on $L_a^2(\mathbb{D})$.

In a similar way, one can define  $\mathcal{W}^*(\phi)$ and
$\mathcal{V}^*(\phi)$ on the Hardy space $H^2(\mathbb{D})$. As mentioned above,
each reducing subspace of $M_\phi$ is exactly the range of some
 projection in $ \mathcal{V}^*(\phi)=\{M_\phi, M_\phi^*\}'$,
 and then  the problem of finding reducing subspaces can be generalized to that of
determining the commutant $\{M_\phi\}'$ of  $M_\phi$ of the Hardy
space.  In this line, one may refer to \cite{A,AB,DW,No,T3}, and
also see \cite{AC,ACR,AD,Cl,Cow1,Cu,JL,Zhu,Zhu2}. In most cases,
multiplication  operators induced by inner functions play a
significant role in commutant problem£¬ just  as done in \cite{BDU,DW,T1}.
 Thomson showed  that if  $h$ is holomorphic on a neighborhood of $\overline{\mathbb{D}}$, then there is always a finite Blaschke product $B$
   and a function $g $ in $ H^\infty(\mathbb{D})$ such that $h=g\circ B $ and
     $\{M_B\}'=\{M_h\}'$, see \cite{T1}. In particular,     $\mathcal{V}^*(h)=\mathcal{V}^*(B)$. It should be pointed out that this assertion is valid not only on the Hardy space but also on the  Bergman space.
It is worthwhile to mention that   when $\eta$ is an inner function, not a M$\ddot{o}$bius transform,
   $M_\eta$ acting on the Hardy space  always has infinitely many minimal reducing subspaces. This follows   from
   the fact that $M_\eta$ is an isometric operator on $H^2(\mathbb{D})$, and hence   $\mathcal{V}^*(\eta)$
    is $*$-isomorphic to $B(H^2(\mathbb{D})\ominus \eta H^2(\mathbb{D}))$, all bounded linear operators on $H^2(\mathbb{D})\ominus \eta H^2(\mathbb{D})$ \cite{No}.

 When we focus on  the Bergman space, the case is different and  difficult. For  a multiplication operator $M_B$  defined by  a finite Blaschke products $B$ with $\deg B\geq 2$,
 one knows that  $M_B$ always  has a  distinguished reducing subspace,  and hence  $\mathcal{V}^*(B)$ is always nontrivial, and
 $\dim \mathcal{V}^*(B) \leq \deg B$, see \cite{HSXY,GSZZ,GH1,Sun}. In the case   $\deg B=2$, it was shown  that   $M_B$   exactly has two distinct minimal
reducing subspaces in \cite{SW} and  \cite{Zhu} independently. Motivated by this fact, Zhu  conjectured that for a finite
Blaschke product  $B$ of degree $n$, there are exactly $n$ distinct minimal reducing subspaces\cite{Zhu}.  In fact, by applying
 \cite[Theorem 3.1]{SZZ2} Zhu's conjecture holds only in very restricted case.
    Therefore, the conjecture
   is modified as follows: $M_B$ has at most $n$ distinct minimal reducing subspaces, and  the number of nontrivial minimal
    reducing subspaces of $M_B$ equals the number of connected components of the Riemann surface
of $B^{-1}\circ B$ on the unit disk \cite{DSZ}.  As verified in \cite{DSZ, GH1},  the modified conjecture is equivalent  to assertion that
 $\mathcal{V}^*(B)$ is abelian.   In the case of $\deg B =3,4,5,6,$ the  modified conjecture
 is demonstrated in \cite{GSZZ,SZZ1,GH1}.  By using the techniques of local inverse and group-theoretic methods,
it was shown  in  \cite{DSZ} that when $\deg B=7, 8,$
 $\mathcal{V}^*(B)$ is abelian. Very recently, an affirmative answer was given for the general case by  Douglas,  Putinar and
   Wang\cite{DPW}. They showed that  $\mathcal{V}^*(B)$ is abelian for any finite Blaschke product.

 However, for an infinite Blaschke product $B$, little is known about the structure of $\mathcal{V}^*(B) $.
 One naturally asks if $\mathcal{V}^*(B) $ always is nontrivial. This is equivalent to the  following   question:
\vskip1.5mm
    \emph{For each infinite Blaschke product $B$, does $M_B$ always have a nontrivial  reducing subspace?}
     \vskip1.5mm
\noindent In view of the case for finite Blaschke products, one may guess that the answer is yes. However, this
paper gives a     negative  answer by constructing a   thin Blaschke product.

  This paper mainly deals with  thin Blaschke products; a  class of infinite Blaschke products.

In this paper, $B$ always denotes a thin Blaschke product, if there
is no other explanation. It is shown that each thin Blaschke product
is a branched covering map, and therefore the method in
\cite{GH2} can be applied to construct  a representation for those
operators in $\mathcal{V}^*(B)$. This then yields  a geometric
characterization for when $\mathcal{V}^*(B)$ is trivial. We find  thin Blaschke products are   $*$-connected. Based on this fact, it is shown  that for ``most" thin Blaschke products $B$, $M_B$ is
irreducible, that is, $M_B$ has no nontrivial reducing subspace. Furthermore,  such a thin Blaschke product $B$ is constructed.
 This is a strong contrast with the case of finite Blaschke products.
  Moreover,  as an application of the methods, we will see that for "most" finite
Blaschke products $\phi$, $M_\phi$ has exactly $2$ minimal   reducing
subspaces. Furthermore, under a mild condition, we give a
geometric characterization for when $M_B$ defined by a thin Blaschke product $B$ admits a nontrivial
reducing subspace.

The paper is organized as follows.

In Section 2,  we  give some properties of thin Blaschke product.
In Section 3, by using the methods in \cite{GH2},  a   representation for those operators in $\mathcal{V}^*(B)$ is presented.
Applying the techniques of analytic continuation and local inverse  gives a geometric sufficient condition for when $\mathcal{V}^*(B)$ is trivial.
In Section 4, by developing    function-theoretic techniques,
 we show that under a mild condition of $B$, $\mathcal{V}^*(B)$ is always trivial and in Section 5,  such a thin Blaschke product $B$ is constructed.
 Section 6 describes  when $\mathcal{V}^*(B)$ is nontrivial by geometry of $B$.
 In Section 7, it is shown  that   $\mathcal{V}^*(B)$ is abelian under a mild condition, which is of interest when contrasted with
  the commutativity    of  $\mathcal{V}^*(\phi)$, where $\phi$ are   finite Blaschke products.

\section{Some properties of thin Blaschke products}
~~~~ In this section,   we mainly show that each thin Blaschke
product is a branched covering map, and  some properties are
obtained.

Before continuing, we need some definitions.
 A holomorphic map \linebreak $\phi:    \mathbb{D}\to  \Omega$ is called \emph{a
branched covering map} if  every point of $ \Omega$ has a connected
open
 neighborhood $U$ such that  each connected component $V$ of $\phi^{-1}(U)$ maps onto
 $U$ by a proper map $\phi|_V$\cite[Appendix E]{Mi}. If in addition, all such maps $\phi|_V$ are biholomorphic maps, then
 $\phi$ is a  holomorphic covering map. For example, a finite Blaschke product $B$ is a branched covering map from $\mathbb{D}$ onto $\mathbb{D}$,
  and moreover,     a proper map.
 It is known that a proper map has some good properties: for instance, it is a \emph{$k$-folds map} for some positive integer $k$ \cite[Appendix E]{Mi}. That is,
 for a proper map $\phi$ and any $w_0  $ in its range, $\phi-w_0$ has exactly $k$ zeros, counting multiplicity.
 This integer $k$  does not depend on the choice of $w_0$.

 On the unit disk, \emph{the pseudohyperbolic  distance} between two points $z$ and $\lambda$
is defined by $$d(z,\lambda)=|\varphi_\lambda(z)|,$$ where $$\varphi_\lambda(z)=
\frac{\lambda-z}{1-\overline{\lambda}z}.$$     A   Blaschke product $B$ is
called \emph{an interpolating Blaschke product}   if its zero
sequence $\{z_n\}$ satisfies $$ \inf_k  \prod_{j\geq 1, j\neq k}
d(z_j,z_k)>0 .
$$ A   Blaschke product $B$ is called \emph{a thin Blaschke
product} \cite{GM1} if    the zero sequence $\{z_n\}$ of  $B$
satisfies
$$
\lim_{k\to\infty} \prod_{j\geq 1, j\neq k} d(z_j,z_k)=1,
$$ Noticing that $\prod_{j\geq 1, j\neq k} d(z_j,z_k)=(1-|z_k|^2)|B'(z_k)| $, the above is equivalent to
 $$\lim_{k\to\infty}(1-|z_k|^2)|B'(z_k)| =1.$$ One
 can show that a thin Blaschke product is never a covering map since its image is exactly the unit disk $\mathbb{D}$\cite[Lemma 3.2(3)]{GM1}.
 However, there do  exist  interpolating Blaschke products
which are also covering maps, 
see \cite[Theorem 6]{Cow1} or \cite[Example 3.6]{GH2}.

Below, we denote by $\Delta(z,r)$ the pseudohyperbolic
disk centered at $z$ with radius $r$;
that is, $$\Delta(z,r)=\{w\in \mathbb{D}:|\frac{z-w}{1-\overline{z}w}|<r\}.$$
 For a Blaschke product $B$, let $B'$ denote the derivative of $B$ and   write the critical value set
  $$\mathcal{E}_B=\{B(z): z\in \mathbb{D} \  \mathrm{and }\   B'(z)=0\}.$$
Let $Z(B)$ and $Z(B')$ denote the zero sets of $B$ and $B'$, respectively.

  We have the following proposition.
\begin{prop} Let $B$ be a  thin Blaschke product. Then both $\mathcal{E}_B$ and $B^{-1}(\mathcal{E}_B)$ are \label{21}
 discrete in $\mathbb{D}$.
\end{prop}
\begin{proof} The proof is divided  into two parts.

 \noindent{\emph{Step 1.}} First we will prove that the critical value set
 $ \mathcal{E}_B $ is discrete in  $\mathbb{D}$.

 It suffices to show that for any $r\in (0,1)$, each $w\in r\mathbb{D}$
  is not an accumulation point of $\mathcal{E}_B$.
   To see this, we will make use of the following statement, see \cite{La,Ga} and \cite[Lemma 4.2]{GM2}:

\emph{  Suppose that $B$ is a  thin Blaschke product with the zero
sequence $\{z_k\}$,
  then for each $ \delta $ with $0<\delta<1$, there exists an $\varepsilon\in(0,1)$ such that $|B(z)|\geq \delta$
  whenever $d(z,z_n)\geq \varepsilon$ for all $n$.}
\vskip1.6mm
  \noindent  Now put $\delta=\frac{1+r}{2}$.
  By the above statement, there exists an $\varepsilon\in(0,1)$ such that $|B(z)|\geq \delta$
  whenever $d(z,z_n)\geq \varepsilon$ for all $n$. Then by a simple argument, one sees that it
  is enough to show that except for finitely many
  $n$, $B|_{\Delta(z_n,\varepsilon)}$ is univalent,  and hence $B'$ has no zero point on $ \Delta(z_n,\varepsilon) $.
   Recall that a holomorphic function $f$ on a domain $\Omega\subseteq
\mathbb{C}$ is called \emph{univalent} if $f$ is injective in
$\Omega$.

      In fact, put $$g_n=B\circ \varphi_{z_n},$$
 and   it is easy to check that
      \begin{equation}
      \lim_{n\to \infty} |g_n'(0)|=\lim_{n\to\infty} \prod_{j\geq 1, j\neq n} d(z_j,z_n)=1. \label{2.1}
      \end{equation}
But a theorem in \cite[p. 171,\, Exercise 5]{Ne} states that, if
$f$ is a holomorphic function from $\mathbb{D}$ to $\mathbb{D}$
 such that
 $$f(0)=0\ \ \mathrm{and} \ \ |f'(0)|=a,$$
 then $f$ is univalent in  $t\mathbb{D}$ with $t\equiv  t(a)=\frac{a}{1+\sqrt{1-a^2}}.$ 
 So each $g_n$ is  univalent on $t(|g_n'(0)|)\mathbb{D}$.
Then by (\ref{2.1}),   there is some natural number $N_0$ such that
$$t(|g_n'(0)|) >\varepsilon, \ n\geq N_0.$$ Thus,
  for all $n\geq N_0$, $g_n$ are
univalent on $ \varepsilon\mathbb{D}$; that is, $B$ is univalent on
$\Delta(z_n,\varepsilon)$. Therefore $\mathcal{E}_B $  is discrete
in  $\mathbb{D}$.

 \noindent{\emph{Step 2.}} Next we show that $B^{-1}(\mathcal{E}_B)$ is discrete in the open  unit disk;
equivalently, $B^{-1}(\mathcal{E}_B)\cap  r\mathbb{D}$ is finite
  for any $0<r<1$.

Now fix such an $r$.
    The  discrete property of $\mathcal{E}_B $
shows that $B(r\mathbb{D})\cap \mathcal{E}_B $ is finite, and hence
$B^{-1}(\mathcal{E}_B)\cap  r\mathbb{D}$  is discrete.
   By arbitrariness of  $ r$,
   $B^{-1}(\mathcal{E}_B)$ is discrete   in
  $\mathbb{D}$,  as desired. The proof is complete.
\end{proof}

Later, we will see that $Z(B')$   always contains infinite points,
say \linebreak $w_1,\cdots,w_n,\cdots$. Since $\mathcal{E}_B$  is
 discrete in $\mathbb{D}$, then   $\lim\limits_{n\to \infty }|B(w_n)|=1$.

 Below, we will show that each thin Blaschke product is a branched covering
 map. First we need a lemma, which is interest in itself. 
\begin{lem} Suppose that $\phi$ is a holomorphic function over $\mathbb{D}$ and $\Delta$ is a  domain in $\mathbb{C}$.
If $\Delta_0$ is a component of $\phi^{-1}(\Delta)$ satisfying
\label{22}
 $\overline{\Delta_0}\subseteq \mathbb{D}$,
then $\phi|_{\Delta_0} :\Delta_0\to\Delta$ is a proper map  onto
$\Delta$.
\end{lem}
\begin{proof} First we show that $\phi|_{\Delta_0} :\Delta_0\to\Delta$ is a proper map.
To see this, assume conversely that there is a sequence $\{z_n\}$ in
$\Delta_0 $ with no limit point in $\Delta_0$, such that
$\{\phi(z_n)\}$ has a limit point $w_0\in \Delta$. Without loss of
generality, assume that $z_n$ converges to some $z_0\in \partial
\Delta_0 $ and
$$\lim_{n\to\infty} \phi(z_n)=w_0.$$
 Since $\overline{\Delta_0}\subseteq \mathbb{D}$,
$z_0\in \mathbb{D}$. Then by continuity of $\phi$, $\phi(z_0)=w_0$,
which shows that there is a connected neighborhood $\Delta_1$ of
$z_0$ such that
$\Delta_1\subseteq \phi^{-1}(\Delta)$.
Notice that $\Delta_0\subsetneqq \Delta_0\cup\Delta_1$, which is a
contradiction to the assumption that $\Delta_0$ is a component of
$\phi^{-1}(\Delta)$. Therefore  $\phi|_{\Delta_0}
:\Delta_0\to\Delta$ is a proper map.

Next we show that $\phi|_{\Delta_0} :\Delta_0\to\Delta$ is onto.
Otherwise, there is a point $w_1\in \Delta-\phi(\Delta_0).$ Since
$\Delta$ is connected, then there is a path $l:[0,1]\to \Delta$ such
that $l(0)\in \phi(\Delta_0)$ and $l(1)=w_1$. Let $t_0$ be the
supremum of
$$\{s\in [0,1]; l(t)\in \phi(\Delta_0)  \mathrm{\ \,for\, all\,\ }  t\in [0,s]\}.$$
By an argument of analysis, $l(t_0)\in \partial \phi(\Delta_0) $.
Write $w_2=l(t_0),$ and clearly  there is a sequence $\{z_n'\}$
in $\Delta_0$ such that $$\lim_{n\to \infty}\phi(z_n')=w_2.$$ Since
$w_2\in\partial \phi(\Delta_0)$, it follows that $\{z_n'\}$ has no limit
point in $\Delta_0$. But $\{\phi(z_n')\}$ has a limit point $w_2\in
\Delta$, which is a contradiction to the statement that
$\phi|_{\Delta_0} :\Delta_0\to\Delta$ is a proper map. Therefore
$\phi|_{\Delta_0} :\Delta_0\to\Delta$ is onto.
\end{proof}

\begin{prop} All thin Blaschke products are branched covering maps. \label{23}
\end{prop}
\begin{proof} Suppose that  $B$ is a thin Blaschke product, $\{z_n\}$ being its zero sequence.
By \cite[Lemma 3.2]{GM1},   $B(\mathbb{D})=\mathbb{D}$.  Now we will
show that $B: \mathbb{D} \to \mathbb{D}$ is a branched covering map.
For each $z\in \mathbb{D}$, set $\delta=\frac{1+|z|}{2}$. Clearly,
$z\in \Delta(0,\delta).$ To finish the proof, we must show that each
connected component of $B^{-1}(\Delta(0,\delta))$ maps onto
$\Delta(0,\delta)$ by a proper map, i.e. by the restriction of $B$
on $\mathcal{J}$.

To see this, notice that by an inclined statement  below Proposition
\ref{21}, there is  an $\varepsilon\in(0,1)$ such that
\begin{equation}
 B^{-1}(\Delta(0,\delta))\subseteq \bigcup_n \Delta(z_n,\varepsilon). \label{2.2}
 \end{equation}
We first study when two pseudohyperbolic disks
$\Delta(z_n,\varepsilon)$ and
 $\Delta(z_k,\varepsilon)$ have   common points.
 If $\Delta(z_n,\varepsilon)\cap  \Delta(z_k,\varepsilon)$ contains a point, say $z'$, then
 by \cite[p.4, Lemma 1.4]{Ga}
 \begin{equation}d(z_n,z_k)\leq \frac{d(z_n,z')+d(z',z_k)}{1+d(z_n,z')d(z',z_k)}\leq d_\varepsilon, \label{2.3}\end{equation}
 where
 $$d_\varepsilon=\max \{\frac{x+y}{1+xy}:0\leq x,y\leq \varepsilon\}<1.$$
  But by the definition of thin
 Blaschke product, it is clear that
 $$\lim_{n\to\infty} \inf_{ k\neq n} d(z_n,z_k)= 1.$$
Then by (\ref{2.3}), it is not difficult to see that there is a
positive integer $N_0$ such that
 $ \Delta(z_n,\varepsilon)\cap  \Delta(z_k,\varepsilon) (k\neq n) $ is empty whenever $n\geq N_0.$
Therefore, by (\ref{2.2})   the closure  of each component of
$B^{-1}(\Delta(0,\delta))$  is contained in $\mathbb{D}$. Applying
Lemma \ref{22} shows that the restriction of $B$ on each connected
component of $B^{-1}(\Delta(0,\delta))$ is a proper map  onto
$\Delta(0,\delta)$. The proof is complete.
 \end{proof}
 Thin Blaschke products share some good properties. For example,
 Lemma \ref{22} states that the pre-image  of the critical
value set of a  thin Blaschke product is discrete.  Under some mild
condition, the product  of two thin Blaschke products is also a thin
Blaschke product. Moreover, one  has the following,  see \cite[pp.
86,106]{Hof}, \cite[pp. 404,310]{Ga} or
\cite[Lemma 3.2(3)]{GM1}. 
\begin{prop} Let $B$ be a thin Blaschke product, then the followings hold:   \label{24}
\begin{itemize}
\item[$(1)$] Each value in $\mathbb{D}$ can be achieved infinitely many times by $B$.
\item [$(2)$] For  every $ w$ in  $\mathbb{D}$,  $\varphi_w\circ B$ is a thin Blaschke product.
\end{itemize}
\end{prop}
\section{Representation for operators in $\mathcal{V}^*(B)$}
~~~~In this section, we will apply the method in \cite[section
2]{GH2}   to  give a representation for those operators in $\mathcal{V}^*(B) $, where $B$ is a thin Blaschke product.
By using the techniques of analytic continuation and local inverse
\cite{T1,T2,DSZ},   a sufficient geometric condition is presented  for when $\mathcal{V}^*(B) $ is trivial.

Remind that for a thin Blaschke product or a finite Blaschke product
$B$,   $\mathcal{E}_B$ denotes its   critical value set. As done in
\cite{DSZ},  set $$E= \mathbb{D}-B^{-1}
 (\mathcal{E}_B),$$
 and by Proposition \ref{21}  $E$ is a connected open set.

 For a holomorphic function $f$ on $ \mathbb{D}$, if $\rho $  is a map defined on some subdomain $V$ of $\mathbb{D} $ such that $\rho(V) \subseteq \mathbb{D}$
 and  $f(\rho(z))=f(z),z\in V$, then $\rho $ is called\emph{ a local inverse} of $f$ on $V$\cite{T1,DSZ}.  For example,
 $f(z)=z^n \ ( z\in \mathbb{D})$ and $\rho(z)=\xi z \ ( z\in \mathbb{D})$, where $\xi$ is one of the $n$-th root of unit.
 Then $\rho$ is a local inverse  of $f$ on $\mathbb{D}$.

 Before continuing, let us make an observation. For an interpolating Blaschke product $\phi$, by \cite[p. 395, Lemma 1.4]{Ga} there is an $r\in (0,1)$
 such that $\phi^{-1}(r \mathbb{D})$  equals the disjoint union of domains $V_i(i=0,1,\cdots)$, and for each $i$, $\phi|_{V_i}:V_i\to r \mathbb{D}$
  is biholomorphic. Define  biholomorphic maps $\rho_{i}:V_0\to V_i$ by setting $$\rho_{i}=\phi|_{V_i}^{-1}\circ \phi|_{V_0},$$ and clearly
   $\phi\circ \rho_i |_{V_0}=\phi |_{V_0}.$ That is, $\rho_i$ are local inverses of $\phi$ on $V_0$.
 Now for any fixed point $z_0\in E$, write $ \lambda=B(z_0)$ and consider $\phi=\varphi_\lambda \circ B $.  Proposition \ref{24}
  implies that $\phi$ is an
 interpolating Blaschke product. Then by the above observation, we get the following lemma.
 \begin{lem}Let  $B$ be a thin Blaschke product.
  For each $z_0\in E$, there always exists a neighborhood $V$ of $z_0$ and countably infinity functions $\{\rho_i\}_{i=0}^\infty$ that are
   holomorphic on $V$ such that:  \label{31}
  \begin{itemize}
  \item [1.]  $B^{-1}\circ B(z)=\{\rho_i(z):i=0,1\cdots\} $ holds for each $ z\in V$.
  \item[2.] $\rho_i$ are local inverses of $B$ on $V$, i.e., $B\circ \rho_i(z)=B(z), \ z\in V$;
    \item [3.] Each $\rho_i: V\to \rho_i(V)$ is biholomorphic.
   \end{itemize}
  In this case, we say that $V$ admits a complete local inverse $\{\rho_i\}_{i=0}^\infty$ of $B$.
   \end{lem}
Notice that in Lemma \ref{31}, $\rho_i( V)\cap \rho_j( V)$ is empty
if $i\neq j$. In this paper, $\rho_0$ always denote the identity
map, i.e. $\rho_0(z)=z,z\in V$.

 Using Lemma \ref{31} and the proof of \cite[Theorem 2.1]{GH2}, we will  get a local representation of operators $S$ in $\mathcal{V}^*(B).$ Precisely,
there is a
  vector $\{c_k\}  $  in $l^2 $ satisfying
\begin{equation}
 Sh(z)=\sum_{k=0}^\infty c_k  h\circ \rho_k(z) \rho_k'(z)  , \label{3.1}
 \ h\in L^2_a( \mathbb{D}),\, z\in  V.
\end{equation}
The identity (\ref{3.1}) holds locally on $V$. Below, we will try to
get a global version of (\ref{3.1}).

We need some definitions from \cite[Chapter 16]{Ru1}. In this
paragraph, all functions (such as $f_j,\rho $ and $\sigma$) are
defined on  some subsets of $E$, where $ E= \mathbb{D}-B^{-1}
 (\mathcal{E}_B)$. \emph{A function element} is a  pair $(f,D)$, where
$D$ is an open  disk and $f$ is a holomorphic function on $D$. Two
function elements $(f_0,D_0)$  and $(f_1,D_1)$ are \emph{direct
continuations} if $D_0\cap D_1$ is not empty and $f_0 =f_1 $ holds
on $D_0\cap D_1$. By a \emph{curve} or a path, we mean a continuous
map from $[0,1]$ into  $E$.   By a \emph{loop}, we mean
 a path $\sigma$ satisfying $\sigma(0)=\sigma(1)$. Given a function element $(f_0,D_0)$ and  a curve    $\gamma$ in $E$ with $\gamma(0)\in D_0$,
if there is  a partition of  $[0,1]$:
     $$0=s_0<s_1<\cdots <s_n=1$$
 and function elements $(f_j,D_j)(0\leq j \leq n)$ such that
   \begin{itemize}
  \item [1.]  $(f_j,D_j)$  and $(f_{j+1},D_{j+1})$
are direct continuations for all $j$ with\linebreak $ 0\leq j\leq
n-1  $;
  \item[2.] $\gamma [s_j,s_{j+1}]\subseteq D_j(0\leq j\leq n-1)$ and $\gamma(1)\in D_n$,
       \end{itemize}then
  $(f_n, D_n)$ is called \emph{an analytic
continuation of $(f_0, D_0)$ along   $\gamma$}; and $(f_0, D_0)$
 is said to \emph{admit} an analytic continuation along  $\gamma$. In this case, we write $f_0\thicksim  f_n$. Clearly, this is an equivalence relation and we write
$[f]$ for \emph{the equivalent class} of $f$. In particular, if
$\rho$
 is a local inverse  of $B$ and  $\gamma$ is a curve such that $\gamma(0)\in D(\rho)$, then   $\rho$  always admits an analytic continuation  $\widetilde{\rho}$    along $\gamma$.
 By \cite[Theorem 16.11]{Ru1}, such an  analytic continuation is unique, and one can show that this $\widetilde{\rho}$ is necessarily a local inverse of
 $B$.  Then applying the monodromy theorem  \cite[Chapter 16]{Ru1} shows that each local inverse $\rho $ of
 $B$ can be extended analytically to any simply connected
domain in $E.$ 

 Below, we will seek  a curve $L$
which passes all points in $ B^{-1}(\mathcal{E}_B)$, such that $
\mathbb{D}-L$  is simply connected. In this way, we can give a
global version of (\ref{3.1}).
\begin{lem} Suppose $\mathcal{E}$ is a discrete subset in $ \mathbb{D}$. Then there is a curve $L$  \label{32}
 containing $\mathcal{E}$   such that $  \mathbb{D}-L$  is simply connected.
\end{lem}

 \begin{proof}  Let $\mathcal{E}$ be a discrete set in $ \mathbb{D}$.
 Without loss of generality, we assume that $\mathcal{E}$
 consists of countably infinite points.  Since  $\mathcal{E}$ is discrete, then there is a strictly increasing sequence $\{r_n\}$
  satisfying  $$\lim_{n\to\infty} r_n=1  \ \mathrm{and }\  \mathcal{E} \subseteq  \bigcup  r_n\mathbb{T}.$$
Notice that each circle $r_n\mathbb{T}$ contains only finite points
in $\mathcal{E}$, and then we can draw an infinite
polygon $L$ passing along these $r_n\mathbb{T}$ and through all points in $\mathcal{E}\cap ( r_n\mathbb{T})$. 
 Since  $L$ tends to $\mathbb{T}$ and has no end point on $\mathbb{T}$, $ \mathbb{D}-L$ looks like
a snail, see Figure 1. 
    Thus, it is easy to construct a sequence of simply connected open sets $U_n$ whose union is the snail $ \mathbb{D}-L$ and
$$U_n\subseteq U_{n+1}, \ n=1,2,\cdots.$$Therefore, $ \mathbb{D}-L$  is simply connected.
\end{proof}
 In the proof of Theorem \ref{32}, one sees that  the curve $L$ consists of at most countably
many arcs and $L$ is relatively closed in $ \mathbb{D}$. Here, by an
arc we mean a subset in $\mathbb{C}$ that is $C^1$-homeomorphic to
some segment.
\vskip 2mm
By Proposition \ref{21} and Lemma \ref{32}, there is a curve $L$
 containing $B^{-1}(\mathcal{E}_B)$   such that $  \mathbb{D}-L$  is simply connected. 
 Notice that all $\rho_k$  are  local inverse of $B$, and the monodromy theorem
  shows that all $\rho_k$ are holomorphic in $\mathbb{D}-L$.
 For each $w$ in $  \mathbb{D}-L$,  define
 $$ \pi_w(h)=\{h\circ\rho_k(w)\rho_k'(w)\},\,  h\in L_a^2(\mathbb{D}).$$
 As done in \cite{GH2}, one gets $ \pi_w(h)\in l^2$. Furthermore,  we have the following.
 \begin{thm}Suppose that $B$ is a thin Blaschke product. If $S$ is a  unitary operator on
the Bergman space which commutes with $M_B$, then there is a unique
operator $W:l^2 \to l^2$ such that
 \begin{equation}
 W\pi_w(h)=\pi_w(Sh),\ h\in L_a^2(\mathbb{D}),\ w\in \mathbb{D}-L. \label{3.2}
  \end{equation}
  This $W$ is necessarily a unitary  operator. Moreover, there is a unique
  vector $\{c_k\}  $  in $l^2 $ satisfying         \label{33}
\begin{equation}
 Sh(z)=\sum_{k=0}^\infty c_k  h\circ \rho_k(z) \rho_k'(z), \label{3.3}
 \ h\in L^2_a( \mathbb{D}),\, z\in \mathbb{D}-L .
 \end{equation}
where all holomorphic functions $\rho_k$ satisfy $B\circ
\rho_k|_{\mathbb{D}-L}=B$, and the right side of (\ref{3.3})
converges uniformly on compact subsets of $\mathbb{D}-L .$
\end{thm}
\noindent The proof is similar as that in  \cite[Section 2]{GH2}.
For the uniqueness of the operator $W$, just notice that for each
$w_0 $ in  $E $,  the zero sequence $\{\rho_i(w_0)\}$ of $ B-B(w_0)
$  is an interpolating sequence, see
 \cite[Lemma 3.2(3)]{GM1} or Proposition \ref{24}. Then following the proof in \cite[Section 2]{GH2} gives the the uniqueness of   $W$.

\vskip1.5mm In fact, for a   finite  Blaschke product $B$, such an
operator $W$ in (\ref{3.2}) was   obtained by R.~Douglas, S.~Sun and
D.~Zheng in \cite{DSZ}. Motivated by a result
 in \cite{DSZ},  we can show that  $W$ has a very   restricted form.
 To see this, let $V $ be chosen as in Lemma \ref{31}  and
 put $\mathcal{E} =B^{-1}(\mathcal{E}_B)$ in Lemma \ref{32}. Since $V $ can be shrunk such that $\overline{\rho_i(V)}\cap \overline{\rho_j(V)} $
 are empty whenever $i\neq j$, we may carefully choose a curve $L$ with all
  $L\cap \rho_i(V )$ being empty.

    Notice that each  $\rho_k$ maps $\rho_j(V )$
   biholomorphically onto some $\rho_i(V )$, and hence $\rho_k\circ \rho_j|_{V }$ makes sense:
  locally $\rho_k\circ \rho_j $ equals some $\rho_i$.
Now   consider the ($n+1)$-th row of both sides of (\ref{3.2}). Let
$\{d_k\}\in l^2$  be  the $(n+1)$-th row of $W$, then by
(\ref{3.2}),
\begin{equation}
Sh(\rho_{n}(w))\rho_{n}'(w)= \sum_{k=0}^\infty  d_k h(\rho_k
(w))\rho_k'(w), \ h\in L^2_a( \mathbb{D})         \label{3.4} \
\mathrm{and}\ w\in  V .
\end{equation}
By the above discussion, there is a permutation $\pi_n$ such that
   \begin{equation}
   \rho_k \circ \rho_{n}|_{ V } = \rho_{\pi_n(k)} |_ {V },k=1,2,\cdots,        \label{3.5}
   \end{equation}
   and hence
    \begin{equation}
   \rho_k' \circ \rho_{n}\, \rho_{n}'|_ {V } = \rho_{\pi_n(k)}' |_ {V }.        \label{3.6}
   \end{equation}

Combining (\ref{3.4}), (\ref{3.5}) with (\ref{3.6}) shows that
 for each $h\in L^2_a( \mathbb{D})$,
$$Sh(\rho_{n}(w))\rho_{n}'(w)=
\sum_{k=0}^\infty  d_{\pi_n (k)} h( \rho_k \circ \rho_{n}(w))\rho_k'
\circ \rho_{n}(w) \rho_{n}'(w),\ w\in   V .$$ Therefore,
$$Sh(\rho_{n}(w))=
\sum_{k=0}^\infty  d_{\pi_n(k)} h( \rho_k \circ \rho_{n}(w))\rho_k'
\circ \rho_{n}(w) ,\ w\in  V ,$$ and hence on $\mathbb{D}-L,$
$$Sh(z)=
    \sum_{k=0}^\infty d_{\pi_n (k)} h\circ \rho_k(z) \rho_k'(z).$$
Therefore  by (\ref{3.3})   \begin{equation*}
    \sum_{k=0}^\infty d_{\pi_n (k)} h\circ \rho_k(z) \rho_k'(z)=
    \sum_{k=0}^\infty c_k h\circ \rho_k(z)\rho_k'(z),\ h\in L^2_a( \mathbb{D}),
    \, z\in \mathbb{D}-L.
    \end{equation*}
 By uniqueness of the coefficients, we get
    $$d_{\pi_n (k)}= c_k,k=1,2,\cdots .$$
    Put $\sigma_n=\pi_n ^{-1}(n\geq 1)$, and then $d_{k}= c_{\sigma_n(k)},k=1,2,\cdots .$
Therefore, we get the form of $W$ as follows.
 \begin{thm}   The infinite unitary  matrix  $W$ in Theorem \ref{33} has a very restricted form. Precisely,
   there is a unit vector $\{c_k\}\in l^2$  and permutations   $\sigma_j(j\geq 1)$  of positive integers
   such that
       \[
   W=\left(\begin{array}{ccccc}
   c_1 & c_2 & \cdots & c_n & \cdots  \\
   c_{\sigma_1(1)} & c_{\sigma_1(2)} & \cdots & c_{\sigma_1(n)} & \cdots \\ 
   \vdots & \vdots &  \ddots & \vdots  & \vdots \\
  c_{\sigma_m(1)} & c_{\sigma_m(2)} & \cdots & c_{\sigma_m(n)} & \cdots   \\
  \vdots & \vdots & \ddots & \vdots & \ddots  \\
   \end{array}\right).
   \]
 \end{thm}\
Below,Theorem \ref{33}  will be used to give a geometric characterization for when $\mathcal{V}^*(B)$ is trivial.
 First, we have the following result. When $B$ is a finite Blaschke product, it
is given by \cite[Lemma 5.1]{DSZ}.
\begin{lem} Suppose that $B$ is a thin Blaschke product, and $S $ is an operator in  $\mathcal{V}^*(B)$ having
 the   form as in (\ref{3.3}):\label{35}
 $$ Sh(z)=\sum_{k=0}^\infty c_k  h\circ \rho_k(z) \rho_k'(z),
 \ h\in L^2_a( \mathbb{D}),\, z\in \mathbb{D}-L .$$
 Then $c_i=c_j$ if $ [\rho_i]= [\rho_j]$. In particular, if $c_i\neq 0$, then $$\sharp \{k\in \mathbb{Z}_+;\
 [\rho_k]=[\rho_i]\}<\infty.$$
 \end{lem}\noindent Remind that for each local inverse $\rho$ of $B$,   $[\rho]$ denotes
  the equivalent class of all analytic continuation $\sigma$ of $\rho$ with $D(\sigma)\subseteq E$.
\begin{proof} Suppose that $S $ is an operator in  $\mathcal{V}^*(B)$ given by
$$ Sh(z)=\sum_{k=0}^\infty c_k  h\circ \rho_k(z) \rho_k'(z),
 \ h\in L^2_a( \mathbb{D}),\, z\in \mathbb{D}-L ,$$
 where $\{c_k\} \in  l^2$.
 Since $\{c_k\}$ is in $l^2$, to prove Lemma \ref{35} it is enough to show that \emph{if $\rho_{j}$ and $\rho_{j'}$  are equivalent, i.e. $
 [\rho_{j}]=[\rho_{j'}]$, then $c_j=c_{j'}.$}
 To see this, we will use the technique of analytic continuation. Fix a point $z_0$ in $\mathbb{D}-L $ and  assume that
  $\rho_{j}$ and $\rho_{j'}$  are equivalent. Then there is a loop $\gamma $ in $E$ based at $z_0$ such that
 $\rho_{j'}$ is exactly  the   analytic continuation of $\rho_j$ along  $\gamma$. By (\ref{3.1}),
for each $f\in L_a^2(\mathbb{D})$, the sum
 $$\sum_{k=0}^\infty c_k  f\circ \rho_k(z) \rho_k' (z)$$ make sense for each $z\in \gamma $ and it is locally holomorphic.
 As done in the proof of  \cite[Lemma 5.1]{DSZ}, suppose that along the loop $\gamma$, each $\rho$ in $\{\rho_i\}_{i=0}^\infty$
  is extended analytically to a holomorphic function  denoted by
  $\widehat{\rho}$, that is defined on some neighborhood of $z_0.$
  In particular,   $\rho_{j}$ is extended analytically to   $\rho_{j'} \equiv \widehat{\rho_{j}}$. We have
\begin{equation}
\sum_{k=0}^\infty c_k  f\circ \rho_k(z_0) \rho_k' (z_0)=\sum_{k=0}^\infty \label{a37}
c_k  f\circ \widehat{\rho_k}(z_0) \widehat{\rho_k}' (z_0).
\end{equation}
Since $\{\rho_k(z_0)\}$ is an interpolating sequence, we may choose
a bounded holomorphic function $f$ over $\mathbb{D}$ such that
$$f(\rho_{j'}(z_0))=1 \quad \mathrm{and } \quad f(\rho_k(z_0))=0,\ k\neq j'.$$
Since $\rho_{j'}=\widehat{\rho_j},$ then (\ref{a37}) gives $c_j=c_{j'},$ as desired.
 \end{proof}

 In this paper,  denote by $D(f) $ the domain of definition for   any function $f$. Now define $ G[\rho]$ to
 be the union of all graphs $G(\sigma)$ of members $\sigma$ in the equivalent class $[\rho]$, with $D(\sigma)\subseteq E$; that is,
  $$G[\rho] =\bigcup_{\sigma \in [\rho]} \{(z,\sigma(z)):z\in D(\sigma)\subseteq E\}.$$

  Clearly, $G[\rho]$ is a subset of $E\times E$. Notice that $G[\rho]$ is a Riemann surface since
it contains only points with the  form $(z,\sigma(z))$, where
$\sigma$ are local inverses defined on some neighborhood of some
$z_0$. Also, observe that for any   local inverses $\rho$ and $\sigma $ of
$B$, $ G[\rho]=G[\sigma] $ if and only if $
[\rho]= [\sigma] $; and  either $G[\rho]=G[\sigma]$ or $G[\rho]\cap G[\sigma]$ is
empty. Following \cite{DSZ} and  \cite{GSZZ}, we define
$$S_B\triangleq \{(z,w)\in\mathbb{D}^2; z\in E \,\ \mathrm{and} \,\  B(z)=B(w)\},$$
equipped with the topology inherited from the natural topology of
$\mathbb{D}^2.$
 It is not difficult to see that $S_B$\emph{ is a Riemann surface}  whose components are $G[\rho]$, where $\rho$ are local inverses of $B$.
In this paper, a Riemann surface means a  complex manifold of complex  dimension $1$,
  not necessarily connected.
  The Riemann surface $G[z]=\{(z,z):z\in E\}$ is called \emph{the identity component}, or \emph{the trivial component} of $S_B$.

Define $\pi_1$ and $ \pi_2$ by setting
 $\pi_1(z,w)=z$  and $ \pi_2 (z,w) = w$  for all\linebreak $(z,w)\in \mathbb{C}^2$. Rewrite $\pi_\rho $ for  the restriction  $\pi_1: G[\rho]\to E $.
By the discussion below Lemma \ref{31}, for any local inverse $\rho$
of $B$ and any curve $\gamma \in E$, there must exist a unique
analytic continuation  $\widetilde{\rho}$  of $\rho$ along $\gamma$.
This implies that $\pi_\rho$ is onto.
 From this observation,  for each $z\in E$  put
  $$[\rho](z)=\{w: (z,w)\in G[\rho]\},$$
  which is nonempty.  Furthermore, by Lemma \ref{31} one can show the following.
  \begin{lem}For any  local inverse $\rho$ of $B$,    \label{36} $\pi_\rho :G[\rho]\to E $ is a covering map.
 \end{lem}
 A standard result in topology says that for a covering map $p:Y\to X$ with $X$ arc-connected, the number $\sharp p^{-1}(x)$ does not depend on the choice of $x\in X$,
 called the multiplicity of $p$, or the number of sheets of $p$.
  Therefore, $\sharp (\pi_\rho)^{-1}(z)=\sharp [\rho](z)$,
 is either a constant integer or $\infty$, called \emph{the multiplicity} of   $G[\rho]$, or \emph{the number of sheets}
  of $G[\rho].$ For simplicity, we rewrite $\sharp [\rho]$ for $\sharp [\rho](z)$ to emphasize that it does
  not depend on $z$.

Combining Theorem \ref{33} with Lemmas 3.5 and 3.6 yields the
following.
   \begin{prop} For  a thin Blaschke product $B$,  if there is no nontrivial component $G[\rho]$ of $S_B$ such that  \label{37}
$\sharp [\rho]<\infty,$ then $\mathcal{V}^*(B) $ is trivial.
Equivalently, $M_B$ has no nontrivial reducing subspace.
\end{prop}
  The existence of   thin Blaschke products as above will be established  in Section 5.
   Also, one will see later that
   the sufficient condition for  $\mathcal{V}^*(B) $ being trivial  in Proposition \ref{37} is necessary under a mild setting, see Section 6.

   \begin{rem} All definitions in this section, such as local inverse, analytic continuation,
   also can be carried  to   finite  Blaschke products $\phi$. However, since
   $\sharp \phi^{-1}(\phi(z))=\deg \phi<\infty$ for each $z\in E$, any component $G[\rho]$ of $S_\phi$ must
   have finite multiplicity where $\rho$ is a local inverse of $\phi$.
\end{rem}

\section{Many $M_B$ have no nontrivial reducing subspace}
~~~~In this section,  it is shown that under a mild condition,
  $M_B$ has no nontrivial reducing subspace for a thin Blaschke product $B$. Also, we show that for a finite Blaschke product  $\phi$, in
many cases $M_\phi$ has exactly two  minimal reducing subspaces, and
thus $\mathcal{V}^*(\phi)$ is abelian. Precisely, our main result is
stated as follows.
\begin{thm} Suppose that $B$ is a thin Blaschke product, and there is a finite subset $F$ of $Z(B')$ such that:  \label{41}
\begin{itemize}
\item[(i)] $B|_{Z(B')-F} $ is injective;
\item[(ii)]  Each point in $ Z(B')-F$ is a simple zero of $B'$.
 \end{itemize}
 Then $M_B$ has no nontrivial reducing subspace, i.e. $M_B$ is irreducible. In this case, $\mathcal{V}^*(B)$ is trivial.
\end{thm}
It seems that  either (i) or (ii) rarely fails for a thin Blaschke
product $B$. If $B$ has the form $B=B_1\circ B_2$ for two Blaschke
products $B_1, B_2$ and $2\leq \deg B_2<\infty $, then (i) fails. In
this case,    $\mathcal{V}^*(B)  $ is nontrivial since
$\mathcal{V}^*(B) \supseteq \mathcal{V}^*(B_2),$ see Example
\ref{6.7}.

In what follows, a notion called
\emph{gluable}  will be introduced, which proves to play an important role in the
proof of Theorem \ref{41}.
 Some lemmas  will be also established.

Now given a holomorphic function $\psi$ defined on
$\mathbb{D}$, set
 $$H \triangleq \psi^{-1} \circ \psi\big(Z(\psi')\big) ,$$
 i.e. $$H=  \{z\in \mathbb{D};\  \mathrm{there \ is \ a}
\,  w\  \mathrm{such  \  that} \  \psi(w)=\psi(z)  \ \mathrm{ and} \
\psi'(w)=0    \}.$$ We assume that $H$ is relatively   closed in
$\mathbb{D}$, and thus its complement set $E\triangleq \mathbb{D}-H$
is connected. Special interest is focused on the case when
$\psi=B$ is a thin (or  finite) Blaschke product. In this case,
Proposition \ref{21} says that  $H$ is a discrete subset of
$\mathbb{D}$,
  and hence
 the open set $E$ is connected.
\begin{Def}Given a function $\psi$ as above and any two  points $z_1, z_2 \in E $
with $\psi(z_1)=\psi(z_2) (z_1\neq z_2)$, if there are two paths
$\sigma_1$ and $\sigma_2$ satisfying the following    \label{43}
\begin{itemize}
\item[(i)]$ \sigma_i(0)=z_i  $ for $i=1,2$,  $ \sigma_1(1)= \sigma_2(1) $. 
\item[(ii)]$ \psi(\sigma_1(t))=\psi(\sigma_2(t)), 0\leq t \leq 1$;
\item[(iii)]  For $0 \leq  t <1$, $\sigma_i(t)\in E,i=1,2$,
 \end{itemize}
 then  $z_1$ and $z_2$ are called directly $\psi$-glued, denoted by $z_1 \stackrel{\psi}\simeq z_2$.
  In addition, we make the convention that $z_1$ is  directly $\psi$-glued to itself,
 i.e.  $z_1 \stackrel{\psi}\simeq z_1$.\end{Def}
\begin{rem}
In Definition \ref{43}, the condition (i)  implies that  \linebreak
$\psi'(\sigma_1(1))=0 $. To see this, assume conversely that
$\psi'(\sigma_1(1))\neq 0.$ Let $\rho$  be a local inverse of $\psi$
defined on a neighborhood of $\sigma_1(0)$ and
$\rho(\sigma_1(0))=\sigma_2(0)$. Since $\psi'(\sigma_1(1))\neq 0,$  by (ii) and (iii)  the continuous
map $$(t,\sigma_1(t)) \mapsto \sigma_2(t) , \, t\in [0,1]$$
naturally induces an analytic continuation $ \widetilde{\rho}$ of
$\rho$ along $\sigma_1$, satisfying  \linebreak
$\widetilde{\rho}(\sigma_1(1))=\sigma_2(1)$. Since    $\sigma_1(1)
=\sigma_2(1)$,  locally  $\widetilde{\rho}$ equals the identity map.
By the uniqueness of analytic continuation,
 $\rho$ is also the identity map, forcing $\sigma_1(0) =\sigma_2(0)$, i.e.
$z_1=z_2$. This is a contradiction. Therefore, $\psi'(
\sigma_1(1))=0$.

\end{rem}
 We also need the following.
\begin{Def}   For a subdomain $\Omega$ of $\mathbb{D}$, if there is a finite sequence $ z_0, z_1,\cdots,
 z_k $ in $\Omega$, such that $$z_j \stackrel{\psi}\simeq z_{j+1}, \,\ 0\leq j \leq k-1,$$
     and all related paths $\sigma_i$ in Definition \ref{43}
  can be chosen in $\Omega$,                                   \label{44}
 then $z_0$ and $z_k$ are called $\psi$-glued on $\Omega$, denoted by $z_0\stackrel{\psi|_\Omega }\sim z_k$;
 in short, $z_0\stackrel{ \Omega } \sim z_k$ or $z_0\sim z_k$. If $w \stackrel{ \Omega } \sim z$ for all
  $z\in \psi^{-1}\circ \psi(w)\cap \Omega$, then  $\Omega$ is called $\psi$-glued with respect to $w $.
   In particular, if $\mathbb{D}$ is
  $\psi$-glued with respect to $w$, then   $\psi $ is called gluable with respect to $w$.
 \end{Def}
 \noindent If   $\psi $ is gluable with respect to $w$ for all $w\in
E$, then $\psi $ is called \emph{gluable}. \vskip2mm

By Definition \ref{43}, each univalent function $\psi$ is
gluable. For more examples, write  $B(z)=z^n(n\geq 2)$. Then
the only critical point of $B$ is $0,$  and $E=\mathbb{D}-\{0\}$.
For two different points $z_1, z_2 \in E  $ with $B(z_1)=B(z_2)$,
define $\sigma_i(t)=(1-t)z_i$, $i=1,2$. Then it is easy to check
conditions (i)-(iii) in Definition \ref{43} hold, and thus
$z_1 \stackrel{B}\simeq z_2$. Therefore, $B$ is gluable.  %
In general, we have the following.
\begin{prop}  Each finite Blaschke product $B$ is  gluable. \label{45}
\end{prop}
\noindent Here, we do not assign a point $z_0\in E$, with respect to
which $B$ is gluable. Below,   one will see that the
 gluable property of $B$ does not depend on the choice of $z_0$
in $E$. Remind that $E=\mathbb{D}-B^{-1}(\mathcal{E}_B).$
\begin{proof}Given a   $z_0\in E$, we must show that $B$ is gluable with respect to  $z_0.$ Since for any M$\ddot{o}$bius map $\varphi$,
the  gluable property of $\varphi\circ B$ is equivalent to that of
$B$ (with respect to  $z_0$), then we may assume that $B(z_0)=0$.
Write $\deg B=n+1$ and
 $$B^{-1}\circ B(z_0) =\{z_0,z_1,\cdots,z_n \}.$$
For each $r\in (0,1)$,  consider the components $\mathcal{J}_{i,r}$
of $\{z:|B(z)|<r\}$ containing some $z_i$ $(0\leq i\leq n)$.
Clearly, for   enough small $r>0$, those $\mathcal{J}_{i,r}$ are
pairwise disjoint. Also, we have
$$\{z:|B(z)|<r\}= \bigsqcup_{i=0}^n \mathcal{J}_{i,r}.$$
Since each $\mathcal{J}_{i,r}$ contains exactly one $z_i$, it must
be $B$-glued. Below, we will enlarge $r$ and make induction on
the number of $B$-glued components of $\{z:|B(z)|<r\}$.

By standard analysis, one can show that there exist a $t(0<t<1)$
satisfying the following:
\begin{itemize}
\item[(i)] for any $r$ with $0<r<t $,   the components $\mathcal{J}_{i,r}$  are pairwise disjoint;
\item[(ii)]  there are at least two distinct  $i$ and $i'$ such that $\overline{\mathcal{J}_{i,t}} \cup \overline{\mathcal{J}_{i',t}}$  is connected.
 \end{itemize}
Now we get $m(m<n+1)$ components of $\cup_{i=0}^n \overline{\mathcal{J}_{i,t}}$, and denote  by $\eta$ the
 minimal distance between any two components among them.
  Write $\varepsilon=\frac{\eta}{2}$ and there is a constant $\delta >0$ such that
  $$\{z\in \mathbb{D}:|B(z)|<t+\delta \}\subseteq \cup_i O(\mathcal{J}_{i,t},\varepsilon),$$
  where   $$O(\mathcal{J}_{i,t},\varepsilon) \triangleq\{w\in \mathbb{D}:\,  \mathrm{there\ \  is\ \  a\ \  point}\ \
   z\in \mathcal{J}_{i,t} \ \ \mathrm{such\ \  that}\ \ |z-w|<\varepsilon\}.$$ 
  Then there are  exactly $m$ components of $\cup_i \mathcal{J}_{i,t+\delta},$  and we deduce that
  \vskip1.5mm \emph{if $\Omega$  is a
  component of $\cup_i \mathcal{J}_{i,t+\delta} $ contains $z_j$, then  $\Omega$ is $B$-glued with respect to $z_j$.}
\vskip2mm
  To see this, notice that for any $s(0<s<1),$ $\mathcal{J}_{i,s}$ is simply connected,  and its boundary
   $\partial \mathcal{J}_{i,s}$ must be a   Jordan curve
 satisfying $$\partial \mathcal{J}_{i,s}\subseteq \{z:|B(z)|=s\}.$$
 Then by (ii) there is at least one point $w\in \overline{\mathcal{J}_{i,t}} \cap \overline{\mathcal{J}_{i',t}},$
  see Figure 2. 
 Let $\Omega$ denote the component  of $\cup_i \mathcal{J}_{i,t+\delta} $  that contains $w$.
 Since $$B(E)=\mathbb{D}-\mathcal{E}_B,$$   then by  Proposition \ref{21} $B(E)$ is connected.
  Thus we can pick a simple curve   $\sigma$ such that $\sigma-\{B(w)\}\subseteq B(E)\cap t\mathbb{D}$,
 and $\sigma$  connects $B(z_0)$ with $B(w)$.
 Since $B^{-1}(\sigma-\{B(w)\})$ consists of disjoint arcs,
 there exists at least one curve $\gamma_i$ in $\overline{\mathcal{J}_{i,t}}$ and another $\gamma_{i'}$ in $\overline{\mathcal{J}_{i',t}}$
 such that $$B(\gamma_i)=B(\gamma_{i'})=\sigma ,$$
 where  $\gamma_i(1)=\gamma_{i'}(1)=w$ and
 $$\gamma_i(0)=z_i \ \ \mathrm{and} \ \  \gamma_{i'}(0)=z_{i'}.$$ 
 The above shows that $\Omega$ is  $B$-glued with respect to $z_i$.
 Similarly, any other component of  $\cup_i \mathcal{J}_{i,t+\delta} $ are also $B$-glued, as desired.

 Noticing that $m<n+1$,   we have reduced    the number of those   $B$-glued components.
 The next step is to find some $t'$ such that (i) and (ii) holds, and etc.
After finite steps, we come to the case of $m=2$; and   repeating
the
   above procedure  shows  that there is  some $r\in (0,1)$  such that  $\{z: |B(z)|<r\}$ is itself $B$-glued with respect to  $z_0$; and thus
  $\mathbb{D} $ is also  $B$-glued. In another word, $B$ is gluable. The proof is complete.
\end{proof}
\begin{cor}  Each thin Blaschke product is gluable. \label{46}
\end{cor}
\begin{proof} Assume that $B$ is a thin Blaschke product.
Fix a point $z_0\in E$ and  write
$$B^{-1}(B(z_0))=\{z_n:n=0,1,\cdots\}.$$  We will
show that $z_0\sim z_k $ for each given $k.$

To see this, set $r=\max \{|z_i|:0\leq i \leq k\}$ and put $$t'=\max
\{|B(z)|:|z|\leq r\}.$$ By   Proposition \ref{21}, $\{B(z):z\in
Z(B')\}$ is a discrete subset of $\mathbb{D}$, and then one can take
a number $t$ such that $$t\not\in \{|B(z)|:z\in Z(B')\}\quad \mathrm{and}
\quad t'<t<1.$$ Let $\Omega_t$ denote the component of $\{z:
|B(z)|<t\}$ containing $r\mathbb{D}$.
By  the maximum modulus principle,
it is not difficult to see that $\Omega_t$ is a simply-connected
domain; and by the choice of $t$, the boundary of $\Omega_t$ is
analytic, and hence a Jordan curve. Then by Riemann's mapping
theorem, there is a biholomorphic map $f:\mathbb{D}\to  \Omega_t$,
which extends to a homeomorphism from  $\overline{\mathbb{D}}$ onto
$\overline{\Omega_t}$.

By Lemma \ref{22}, $B|_{\Omega_t} :\Omega_t \to t\mathbb{D} $ is a
proper map, and then $B\circ f: \mathbb{D} \to t \mathbb{D}$ is also
a proper map. By \cite[Theorem 7.3.3]{Ru2}, there is a finite
Blaschke product $B_0$ such that $B\circ f= tB_0$. Then by
Proposition \ref{45} $B_0$ is gluable on $\mathbb{D}$, then so
is $B\circ f$. By the    biholomorphicity of $f$, it follows that
$\Omega_t$ is $B$-glued with respect to $z_0.$ Since $
\{z_0,\cdots,z_k\}\subseteq \Omega_t,$ $z_0\sim z_j$ for $j=0,
\cdots,k.$ Therefore, $$z_0\stackrel{\Omega_t }\sim z_j, j=0,
\cdots,k,$$ and hence $$z_0\stackrel{\mathbb{D} }\sim z_k.$$
By  arbitrariness of  $k,$    $\mathbb{D}$ is $B$-glued  with respect to $z_0$.
 By the arbitrariness of $z_0$, $B$ is gluable.   The proof is complete.
\end{proof}
In what follows, Bochner's theorem\cite{Wa1,Wa2} will be useful.
  For a holomorphic function $\psi$,
if $\psi'(w_0)=0$, then $w_0$ is called \emph{a critical point} of $\psi$.
\begin{thm}[Bochner] If $B$ is a finite Blaschke product with degree $n$, then the critical set $Z(B')$ of $B$ is contained \label{47}
in $\mathbb{D}$, and $B$ has exactly $n-1$ critical points, counting
multiplicity.
\end{thm}
Using Bochner's theorem, one can give the following.
\begin{rem}  For a thin Blaschke product $B$, $Z(B')$ is always an infinitely set.  To see this,  the proof of Proposition \ref{21} shows that for
each integer  $n$ there is a component $\Omega_r$ of $\{|B(z)|:|z|< r\}$
satisfying
$$z_j\in \Omega_r,\, j=0,\cdots,n,$$
and $\overline{\Omega_r} \subseteq \mathbb{D}.$ Applying Lemma \ref{22} shows that $B|_{\Omega_r}: \Omega_r \to r \mathbb{D}$ is a
$k$-folds $(k\geq n+1)$ proper map (for the definition ``$k$-folds", see the 2nd paragraph in Section 2). By the proof of
 Corollary \ref{46}, we may regard $B|_{\Omega_r}$ as a finite
Blaschke product with degree $k$. Then by Bochner's theorem, $(B|_{\Omega_r})' $ has exactly $k-1(k-1\geq n)$ zero points,
counting multiplicity. By the arbitrariness of $n$, $B'$ has infinitely many zeros on $\mathbb{D}$. Also, this gives a picture
describing how  a thin Blaschke product $B$ is obtained by ``gluing" finite Blaschke products.\end{rem}
 Below,  $B$ denotes a thin or finite
Blaschke product.  For any component  $G[\rho]$ of $S_B$, if
$(z_0,w_0)\in G[\rho]$, then \emph{we also write $ [(z_0,w_0)]  $
for $[\rho] $
 and $G[(z_0,w_0)]$  for $G[\rho]$.}

  The following indicates the relation between the gluable property
   and the geometry of components in $S_B$. Its proof will be placed at the end of this
section.
\begin{lem}Let $B$ be a finite or thin Blaschke product.
Given $z_0,z_1$ and $z_2$ in $E$ with $B(z_0)=B(z_1)=B(z_2)$, suppose that there are three paths $\sigma_i$ such that \label{48}
\begin{itemize}
\item[(i)]$ \sigma_i(0)=z_i  $ for $i=0,1,2$,  $ \sigma_1(1)= \sigma_2(1) $, and $B'(\sigma_0(1))\neq 0$;
\item[(ii)]  $  B(\sigma_0(t))=B(\sigma_1(t))=B(\sigma_2(t)), 0\leq t \leq 1$;
\item[(iii)]  For $0 \leq  t <1$, $\sigma_i(t)\in E,i=0,1,2$.
 \end{itemize}
 Then $G[(z_0,z_1)]=G[(z_0,z_2)]$.
\end{lem}

\noindent \textbf{Proof of Theorem \ref{41}.}  First we assume that
$F$ is empty. In this case, we have  the following claim.

\noindent \textbf{Claim:} With the assumptions in Theorem \ref{41}
and $F$ being empty,
 \linebreak $B(z_0)=B(z_1)=B(z_2)$ and $z_1 \sim z_2$($z_1\neq z_2$), then $G[(z_0,z_1)]=G[(z_0,z_2)]$.

Since $ \sim$ gives an equivalent relation, we may assume that
$z_1\stackrel{*}\sim z_2$. Moreover, by the proof of Proposition
\ref{45} and Corollary \ref{46},   it is enough to deal with the
following case:
 there are two different components $\Omega_1$ and $\Omega_2$  of $\{z:|B(z)|<t\}$ such that
$$z_i\in  \Omega_i, \quad i=1,2;$$
and there is a $w\in Z(B')$ such that $w\in \partial \Omega_1 \cap
\partial \Omega_2$, see Figure \ref{fig3}.
 Clearly,  $|B(w)|=t$. Now we can
find a simple curve $\sigma$   connecting $B(z_1)$ with $B(w)$,
satisfying $\sigma[0,1) \subseteq E\cap t\mathbb{D}$. Since
$B^{-1}(\sigma[0,1))$ consists of disjoint arcs, among which the
endpoints are $z_0,z_1$ and $z_2$, respectively. Thus, there are
three curves $\sigma_i(i=0,1,2)$ such  that
\begin{equation}B\circ \sigma_i(t)=\sigma(t),0 \leq t \leq 1, i=0,1,2.       \label{4.1}
\end{equation}
Notice that $\sigma_1(1)=\sigma_2(1)=w$, and then $\sigma_0(1)\neq
w$. Otherwise, by the theory of complex variable  we would have
$B'(w)=B''(w)=0$, which is a contradiction to the assumption that
$B'$ has only simple zeros. Since $B|_{Z(B')}$ is injective, it
follows from (\ref{4.1}) that $B'(\sigma_0(1))\neq 0. $ Then by
Lemma \ref{48}, $G[(z_0,z_1)]=G[(z_0,z_2)]$, as desired.

Next we will draw a graph, $\{z_n\}$ being all the  vertices. For
some preliminaries of
  the graph theory, one may refer to \cite[pp. 83-86]{Ha}.
  By the
above claim, when $z_j\stackrel{*}\sim z_k$, we have \begin{equation}G[(z_i,z_j)]=G[(z_i,z_k)],
\quad \mathrm{if} \quad z_i\neq z_j\quad  \mathrm{and } \quad
z_i\neq z_k. \label{4.2}\end{equation}
For the pair $(z_j,z_k)$ with $ z_j\stackrel{*}\sim z_k $, if  those
$\sigma_i$ in Definition  \ref{43} can be carefully chosen so that they
pass  no other point $z_l$, then we draw an abstract edge
between $z_j$ and $z_k$, called an \emph{abstract}  $*$-arc. All
different $*$-arcs  are assumed to be \emph{disjoint}, which means
that they have no intersection except for possible endpoints.
 This can be done because all $z_j$ are in $\mathbb{C}$, which can be regarded as a subspace of $\mathbb{R}^3$; and
all $*$-arcs are taken as  disjoint  curves in  $\mathbb{R}^3$.
 By the proof of Proposition \ref{45} and Corollary \ref{46}, such $*$-arcs yields an abstract path
 (consisting of several adjoining $*$-arcs)
  connecting any two given points in $\{z_n\}$, see Figure \ref{fig4}.
Therefore, all   points $z_n$ and $*$-arcs consists of a connected
graph, $S$. We may assume that   $S$ is simply-connected, since by
the graph theory any connected graph contains a contractible (and
thus simply-connected) subgraph having all vertices of $S$, called
 a maximal tree\cite[\mbox{p. 84,} Proposition 1A.1]{Ha}.
  For each connected subset $\Gamma\subseteq \{z_n\}$
(with some  $*$-arcs ) and some $z   \in  \{z_n\}-\Gamma$, from
(\ref{4.2}) it is not difficult to see that $G[(z,z')]=G[(z,z'')]$
whenever $z',z''\in \Gamma.$ Then we may write $G[(z,\Gamma )] $ to
denote $G[(z,z')]$, and take care that  $G[(z,\Gamma )] $ may
contain  $(z,z_n)$ for some $z_n\not\in \Gamma.$

Now pick one $w \in \{z_n\}$. In $S$, delete $w$ and all $*$-arcs
beginning at $w$, and the remaining consists finitely or infinitely
many connected components $\Gamma_i$. If each component  consists of
infinite $z_j$, then except for the identity component
$G[(w,w)]\equiv G[ z]$, all other components $G[(w,\Gamma_i)]$ of
$S_B$ are infinite; that is,  $\sharp G[(w,\Gamma_i)]=\infty.$
 Otherwise, there is one component $\Gamma^*$
consisting of only finite points. Since $S$ is a tree and $\Gamma^* \subseteq S$, then
$\Gamma^*$ is also a tree. There are two cases under consideration.

\noindent \textbf{Case I}  $\sharp \Gamma^*\geq 2$. In this case,
 one can find an endpoint $w^*$ of $\Gamma^*$. Replacing $w$ with $w^*$,   the above
paragraph shows that there are only two components: the identity
component and   $G[(w^*,\Gamma)]$, which must be infinite. In either
case, the only finite component of $S_B$ is the
identity component. Then by Proposition  \ref{37}  $ \mathcal{V}^*(B)
$ is trivial, and  $M_B$ has no nontrivial reducing subspace.

\noindent \textbf{Case II} $\sharp \Gamma^*=1$. In this case, write $\Gamma^*=\{w^*\}$
and replace $w$ with $w^*$. By the reasoning as in Case I, one can deduce that $ \mathcal{V}^*(B)
$ is trivial, and  $M_B$ has no nontrivial reducing subspace. Thus, the case of $F$  being empty is done.
\vskip2mm
In general, $F$ is not empty. In this case,  also we have the
connected graph $S$. From the proof of Proposition \ref{45} and Corollary \ref{46},
one can construct a connected  sub-graph $\Gamma_0$ (related with $F$)  of $S$ such that
 \begin{itemize}
\item[(1)]  $\Gamma_0$  contains only  finite vertices;
\item[(2)]  in the case of $F\neq \emptyset,$   the above claim  also holds provided that $z_0 \not\in\Gamma_0.$
 \end{itemize}
   Now regard $\Gamma_0$ as
one ``point'' or   a whole body.
 Delete  $\Gamma_0$ and consider the remaining.  If there is some component
 containing   only finite points, or there are at least two different $\Gamma_i$ with $\sharp \Gamma_i=\infty$,
  then by similar discussion as above, $ \mathcal{V}^*(B) $ is trivial. The only remaining  case is that,
$S$ looks like a half-line, with $\Gamma_0$ the endpoint.
 Assume conversely that $ \mathcal{V}^*(B) $ is nontrivial. Then by Proposition  \ref{37}, $S_B$ must contain  one nontrivial component
of finite multiplicity. Let $w$ be the  point next to $  \Gamma_0  $
in $S$, the only possible components of $S_B$ are the following: the identity component $G[(w,w)]$, one infinite component
and one
component $G[(w,\Gamma_0)]$ of finite  multiplicity \linebreak $n=\sharp \Gamma_0$, see Figure 5.
  However, if we let $w'$ be  the  point next to $w$, then the
only nontrivial finite component
 $$G[(w',\Gamma_0)]=G[(w', \Gamma_0\cup \{w\})]$$ whose  multiplicity is not less than
$\sharp  \Gamma_0+1$, which is a contradiction. Therefore, $S_B$ has exactly $2$ components:  the identity component $G[(w,w)]$ and
one infinite component. Thus, $ \mathcal{V}^*(B) $ is trivial, and hence $M_B$ has no nontrivial reducing subspace. The proof is
complete.
 $\hfill \square$
\vskip1.5mm
Let us have a look at the above proof.
 In the case of a thin Blaschke product, we   present  a finite subset $\Gamma_0$ of $S$, that is related to a finite set $F$. However,
if $B$ is a finite Blaschke product, it is probable that $\Gamma_0=S,$ which would spoil the proof of Theorem  \ref{41}.

In fact, for a finite Blaschke product $B$,   $S$  is
a graph of finite vertices. In the case of  $F$ being empty,
applying  the proof of Theorem \ref{41} shows
 that there is a point $w$ and a connected subgraph $\Gamma$ such that $S$ consists of $\Gamma$ and $w$.
 Then $S_B$ has exactly two components: $G[z]$ and $G[(w,\Gamma)],$ which are necessarily of finite multiplicities since
 $B$ is of finite order. Thus, we get the following result.

\begin{cor}
For a finite Blaschke product  $B$, if the   conditions  in Theorem
\ref{41} hold with $F$ being  empty, then $M_B$ has exactly two minimal     \label{42}
reducing subspaces,  and thus $\dim
\mathcal{V}^*(B)=2$. Therefore, $\mathcal{V}^*(B)$ is abelian.
\end{cor}
\noindent Very Recently, Douglas, Putinar and Wang \cite{DPW} showed that   $\mathcal{V}^*(B)$
is abelian and is $*$-isomorphic to   $q $-th direct sum
$\mathbb{C}\oplus \cdots \oplus \mathbb{C}$ of  $\mathbb{C}$, where
\linebreak $q$ $(2\leq q\leq n)$ is the number of components of  the Riemann surface $S_B$.
  Corollary  \ref{42}  shows that
in most cases $q=2.$ In this case, write \linebreak
$M_0=\overline{span\{B'B^n:n=0,1,\cdots\}}$, and
then $M_0$ and $M_0^\perp$ are the only minimal reducing subspaces for $M_B$ \cite[Theorem 15]{SZ2}.

To end this section,  the proof of Lemma \ref{48} will be presented.  We
need B$\ddot{o}$ttcher's theorem, which is of independent interest,
see \cite[Theorem 9.1]{Mi} for example.
\begin{thm}[B$\ddot{o}$ttcher]  Suppose that $$f(z) = a_nz^n + a_{n+1}z^{n+1}+ \cdots,$$
where  $a_n \neq  0$ for $n\geq 2$. Then there exists a local
holomorphic change of coordinate $w = \varphi(z)$ which conjugates
$f$ to the $n$-th power map $w \mapsto w^n$ throughout some
\label{49} neighborhood of $\varphi(0) = 0.$ Furthermore,  $\varphi$
is unique up to multiplication by an ($n-1$)-st root of
unity.\end{thm}

Roughly speaking, Theorem \ref{49} tells us that if $f$ is a
non-constant function that is holomorphic at $z_0$  and $f'(z_0)=0
$, then on an enough small neighborhood of $z_0$, it behaves no more
complicated than the map $h(z)= z^n$ at $z=0$. \vskip2mm \noindent
\textbf{Proof of Lemma \ref{48}.} Suppose that those conditions
(i)-(iii) hold. Without loss of generality, we assume that
$\sigma_1(1)=0=B(\sigma_1(1)).$ From (i) and the comments below
Definition \ref{43}, we have  $B'(\sigma_1(1))=0$.
 Then  by B$\ddot{o}$ttcher's theorem,
there is local holomorphic change of coordinate $w = \varphi(z)$
defined on a neighborhood of $  \sigma_1(1) $, such that
$\varphi\circ B \circ \varphi^{-1}(w)=w^n$. This implies that there
are $n$ disjoint paths $$\gamma_1,\cdots,\gamma_n$$ whose images are
contained in a small neighborhood of $0$ (see Figure 6),
 and a loop $\sigma$ on a
neighborhood of $\sigma_0(1) $ such that
 \begin{itemize}
\item[(a)]   $\gamma_j(1)=\gamma_{j+1}(0),j=1, \cdots, n-1 \quad \mathrm{and} \quad \gamma_n(1)=\gamma_1(0)$;
\item[(b)]  $  B(\sigma_0(t))=B(\gamma_j(t)) , 0\leq t \leq 1, j=1,\cdots, n$;
\item[(c)]  for some  enough large $t^*\in (0,1)$,  $\sigma_0(t^*)=\sigma(0)$,
  $\sigma_1(t^*)=\gamma_1(0)$ and $\sigma_2(t^*)=\gamma_{j_0}(1)$ holds for some
$j_0(1\leq j_0 \leq n)$.
 \end{itemize}
 Now  let  $\widehat{\sigma_i}$ denote the segment of the loop
$\sigma_i$ connecting $\sigma_i(0)$ with $\sigma_i(t^*)$ for
$i=0,1,2$. Precisely, put 
  $$\widehat{\sigma_i}(t)=\sigma_i(t^*t),\, 0\leq t\leq 1,\ \ i=0,1,2.$$ Notice that for any
$j$, there always exist an  $m$ such that $$\sigma^m (t)\mapsto
\big(\gamma_j\cdots  \gamma_1 \big)(t)$$ induces a natural analytic
continuation.
 Let $m_0$ be the integer corresponding to $j=j_0$ and consider the loop
$\widetilde{\sigma} \triangleq \widehat{\sigma_0}^{-1}\sigma^{m_0}
\widehat{\sigma_0}$. Also, there is a path
 $$\widetilde{\gamma} \triangleq \widehat{\sigma_2}^{-1}\gamma_{j_0}\cdots \gamma_2 \gamma_1 \widehat{\sigma_1} $$ which connects
 $z_1$ with $z_2$.
 Then one can show that $\widetilde{\sigma} (t) \to \widetilde{\gamma}(t)$ naturally gives an analytic  continuation.
 Since $G[(\widetilde{\sigma} (0), \widetilde{\gamma}(0) )]=G[(\widetilde{\sigma} (1), \widetilde{\gamma}(1) )]$,
and $$\widetilde{\sigma} (0) =\widetilde{\sigma} (1)=z_0,
\widetilde{\gamma}(0)=z_1,\ \ \mathrm{and}\ \
\widetilde{\gamma}(1)=z_2,$$ we have    $G[(z_0,z_1)]=G[(z_0,z_1)]$.
The proof of Lemma \ref{48} is complete.
 $\hfill \square$

\section{The construction of $B$ such that $M_B$ has no nontrivial reducing subspace}
~~~~In this section, we will construct a thin Blaschke product $B$
such that $M_B$ has no nontrivial reducing subspace.

First, we establish the following lemma.
\begin{lem}  Suppose that $B$ is a finite  Blaschke product with $\deg B\geq 2$.
Then for any $r\in (0,1)$, there is an  $s(0<s<1)$ such that for any
\label{51} $\lambda\in \mathbb{D}$ with $|\lambda|>s$, we have the followings:
\begin{itemize}
\item[(i)] there is a critical point $w^*$ of   $ B\varphi_\lambda $ such that $ |B(w^*)\varphi_\lambda(w^*) |>r$.
   \item[(ii)]  $(B\varphi_\lambda)(w^*)$ is different from the values of $ B\varphi_\lambda $
 on $Z\big((B\varphi_\lambda)' \big)-\{w^*\}$.
 \end{itemize}
\end{lem}
\noindent In particular, with an appropriate choice of $s$, if the
restriction $B|_{Z(B')}$ of $B$ on $Z(B')$ is injective, then the
restriction of $B\varphi_\lambda$ on $Z\big((B\varphi_\lambda)'
\big)$ is also injective.

\vskip2mm Before we give the proof, let us make an observation.
 For each $\lambda \in \mathbb{D}-\{0\}$  and a compact subset $K$ of
$\mathbb{D}$,
\begin{eqnarray*} | \varphi_\lambda(z)-\frac{\lambda}{|\lambda|}| &=&|(\frac{\lambda-z}{1-\overline{\lambda}z} -\lambda)+ (\lambda- \frac{\lambda}{|\lambda|})|\\
&\leq &  | z\frac{1-|\lambda|^2}{1-\overline{\lambda}z} | + 1-|\lambda| \\
&\leq &  2\frac{ 1-|\lambda| }{1-|z|} \leq   \frac{2(1-|\lambda|)}{1-r(K)},
\,z\in K ,\end{eqnarray*} where  $r(K)\triangleq \max \{|z|:z\in
K\}$. Write  $ \psi_\lambda=\frac{|\lambda|}{\lambda} \varphi_\lambda$, and then
\begin{equation}|\psi_\lambda(z)-1| \leq   \frac{2(1-|\lambda|)}{1-r(K)}, \,z\in K .\label{5.1} \end{equation}

\noindent \textbf{Proof of Lemma \ref{51}.} Suppose that $B$ is a
finite  Blaschke product with $\deg B \triangleq n \geq 2$. It
suffices to show that for  any $r\in (0,1)$, there is an $s(0<s<1)$
such that  (i) holds for all $\lambda\in \mathbb{D}$ with $|\lambda|>s$. If this
has been done, we will show that $s$ can be enlarged such that (ii)
also holds. To see this, given   an  $r\in (0,1)$, there is an $s$
such that (i) holds for all $\lambda\in \mathbb{D}$ $(|\lambda|>s)$. Pick an
$r'\in (r,1)$ satisfying $$r'>\max \{|z|; B'(z)=0\},$$ and 
then find an $r''\in (0,1)$ such that $$r''>\max \{|B(z)|:|z|\leq
r'\}\quad \mathrm{and} \quad r''>r'>r.$$
 For this $r''$, there is an  $s''(0<s''<1)$ such that  (i) holds for all $\lambda\in \mathbb{D}$ with $|\lambda|>s''$.
On the other hand, by (\ref{5.1}) there is an $s'\in (0,1)$ such
that the following holds: for any $\lambda \in \mathbb{D}(|\lambda|>s')$,
$$\|B \psi_\lambda-B \|_{r''\overline{\mathbb{D}},\infty}$$
is enough small so that
$$\|(B \psi_\lambda)'-B' \|_{r' \mathbb{T} ,\infty}< \min \{|B'(z)|: |z|=r'\}.$$
Then applying Rouche's theorem  implies that on $r'\mathbb{D}$,  $
\big(B  \psi_\lambda \big)' $ has the same number of zeros
as $B'$, counting multiplicity. Since Bochner's theorem (Theorem \ref{47}) 
says that
     each Blaschke product with $k$ zeros has exactly $k-1$
critical points in the unit disk $\mathbb{D}$, counting
multiplicity, then   $ \big(B  \psi_\lambda \big)' $ has $n-1$ zeros on
$r'\mathbb{D}$. Again by Bochner's theorem,  $ \big(B  \psi_\lambda \big)'
$ has $n$ zeros in $\mathbb{D}$, and hence there is exactly one
remaining zero $w^*$ outside
  $r'\mathbb{D}$. That is,
  $$Z\big((B  \psi_\lambda )'\big)\cap r'\mathbb{D}= Z\big((B  \psi_\lambda )'\big)-\{w^*\},$$
which gives that \begin{eqnarray*} r'' &>&  \max \{|B(z)|:|z|\leq r'\}\\
&\geq &  \max \{|B\psi_\lambda(z)|:|z|\leq r'\} \\
&\geq &  \max \{|B\psi_\lambda(z)|:z\in  Z\big((B  \psi_\lambda )'\big)-\{w^*\}
\}.\end{eqnarray*} By our choice of $s''$, for any $\lambda\in \mathbb{D}$
with $\lambda>\max\{s',s''\}$,
$$|B(w^*)\psi_\lambda(w^*)|>r''\geq   \max \{|B\psi_\lambda(z)|:z\in  Z\big((B  \psi_\lambda )'\big)-\{w^*\} \}.
$$
Also, $|B(w^*)\psi_\lambda(w^*)|>r$ since $r''>r$. Therefore, we can pick
an \linebreak $s=\max\{s',s''\}$ such that both (i) and (ii) hold
for any
 $\lambda\in \mathbb{D}$ with $|\lambda|>s$.

It remains to  prove (i).  Set  $t_0=\max\{|z|: z\in Z(B)\}$ and
$$t_1=\max \{|B(z)|; z\in t_0\overline{\mathbb{D}}\}.$$  For any
$t>t_1$, let $\Omega_{t}$ denote the component of $\{z:|B(z)| <t\}$
containing $t_0\mathbb{D}$. We will see that $$\Omega_{t}=\{z:|B(z)|
<t\}.$$The reasoning is as follows. By Lemma \ref{22},
$B|_{\Omega_{t}} : \Omega_{t}\to t\mathbb{D}$ is a proper map. By
the property of proper map\cite[Appendix E]{Mi}, $B|_{\Omega_{t}}$
is a $k$-folds map for some integer $k$. Since the restriction
$\big(B-0\big)|_{\Omega_{t}}$ has $n(n=\deg B)$ zeros, then $k=n$.
This implies that there is no component of $\{z:|B(z)| <t\}$
 other than $\Omega_{t}$. Therefore, $\Omega_{t}=\{z:|B(z)| <t\}$.

Pick a $t^*\in ( 0,1) $ such that   $$t^*>\max \{|z| :z\in
Z(B')\}\quad  \mathrm{and} \quad t^*>t_1.$$ Without loss of
generality, we assume that $r>t^*.$ Write $r_0=r$, and pick $r_1$
and $r_2$ such that $$t^*<r_0<r_1<r_2<1.$$ As mentioned above,
$\Omega_{r_j}=\{z:|B(z)| <r_j\}(j=0,1,2).$ Since $$\|B\psi_\lambda-B
\|_{\Omega_{r_2},\infty}\leq \| \psi_\lambda-1 \|_{\Omega_{r_2},\infty},$$
then by (\ref{5.1})  for   enough small $\varepsilon>0$, there
exists  an $s\in (r(\overline{\Omega_{r_2}} ),1)$ such that
 \begin{equation}\| \psi_\lambda-1 \|_{\Omega_{r_2},\infty}<\varepsilon , \quad |\lambda|>s. \label{a52}\end{equation}
  Let $\Omega_1$ be the component  of $$\{z: |B(z) \psi_\lambda(z)|<r_1\}$$ containing $Z(B)$,
   and  $\Omega_2 $  denotes the component containing $\lambda$, see Figure 7.
 The prescribed  $\varepsilon$ in (\ref{a52}) can be chosen   enough small so that
 $$\Omega_{r_0} \subseteq \Omega_1   \subseteq  \Omega_{r_2}.$$
 Since $|\lambda|>s> r (\overline{\Omega_{r_2}} )$,
then  $\Omega_1\neq \Omega_2 $, and hence $\Omega_1\cap \Omega_2 $
is empty. By Lemma \ref{22} the restrictions of $ B \psi_\lambda$ on
$\Omega_1$ and $\Omega_2$ are proper maps onto $r_1 \mathbb{D}$.
Notice that $B \psi_\lambda  |_{\Omega_1}$  is of $n$-folds, and then $ {B
\psi_\lambda}|_{\Omega_2}$ is a $1$-fold map. That is, $ {B
\psi_\lambda}|_{\Omega_2}$ is univalent, and hence the derivative of $B(z)
\psi_\lambda$  never varnishes  on $\Omega_2$. By Bochner's theorem,
the derivative of $B(z) \psi_\lambda$ has $n$ zeros, with $n-1$ in
$\Omega_{r_2}$ since $B'$ has $n-1$ zeros  in $\Omega_{r_2}$.  Thus,
there is exactly one critical point $w^*$ outside $\Omega_{r_2}
\sqcup  \Omega_2$, forcing $w^* \not\in \Omega_1 \sqcup  \Omega_2 $.
 Since $$\{z: |B(z) \psi_\lambda(z)|<r_1\}=\Omega_1 \sqcup  \Omega_2,$$  $|B(w^*)
\psi_\lambda(w^*)|\geq r_1>r$, as desired. The proof is complete. $\hfill
\square$

\vskip2mm
Below, we write  $\stackrel{\circ} D$  for the interior of a set $D$. Combining
Rouche's theorem with Cauchy's formula for derivatives, we have the
following.
\begin{lem}   Suppose that $B$ is a finite  Blaschke product and
 $$Z(B')\subseteq   \stackrel{\circ}D_1 \subseteq  D_1\subseteq   \stackrel{\circ} D_2  \subseteq    D_2 \subseteq  \mathbb{D},$$ \label{52}
  where  $D_1$ and $D_2$ are two closed disks.
  Then for any enough small $\varepsilon>0$, there is a $\delta>0$ such that for any holomorphic  function $h$ over $\mathbb{D}$ with
  $\|h-B\|_{D_2,\infty}<\delta$, we have
  $Z(h')\cap D_1 \subseteq O(Z(B'), \varepsilon)\subseteq   \stackrel{\circ}D_1 $.

  If in addition, all zeros of  $ B' $ are simple zeros, then all points in \linebreak $Z(h')\cap D_1$
  are simple zeros of $h'$; and if $B|_{Z(B')} $ is injective, then the restriction of $h $
  on $Z(h')\cap D_1$ is also injective.
\end{lem}
By applying Lemmas \ref{51} and \ref{52}, we are able to give  the
following.
\begin{thm}For a sequence $\{z_n\}$ in $\mathbb{D}$ satisfying $\lim\limits_{n\to \infty} |z_n|=1$, \label{53}
 there is a  Blaschke subsequence $\{z_{n_k}\}$ such that
 $M_B$ has no nontrivial reducing subspace, where $B$ denotes the    Blaschke product for $\{z_{n_k}\}$.
\end{thm}
Let us make an observation.  For   $x\in (0,\frac{1}{2})$,
we have $e^{2x}<1+4 x$. Then $$\prod_{n=0}^\infty
(1+\frac{x}{2^n})= \exp(\sum_{n=0}^\infty  \ln (1+\frac{x}{2^n}))<\exp(\sum_{n=0}^\infty   \frac{x}{2^n})
< 1+4x.$$ Take $x= \dfrac{\varepsilon}{2}$, and then
\begin{equation}\prod_{n=1}^\infty (1+\frac{\varepsilon}{2^n})<1+  2\varepsilon, \, 0< \varepsilon <1. \label{5.2}
\end{equation}
Now let $K$ be a compact subset  of $\mathbb{D}$, and $0<\delta<1.$ If $\{\lambda_n\}$ is a Blaschke
sequence  satisfying
\begin{equation} 1-|\lambda_n|\leq \frac{\delta (1-r(K))}{2^{n+2}} ,        \label{5.3}
 \end{equation}
then by (\ref{5.1}) we get $$ |\psi_{\lambda_n}(z)-1| \leq \frac{\delta}{2^{n+1}},\, z\in K.$$
Since   $\prod_n \psi_{\lambda_n}$ defines a Blaschke product, then
by (\ref{5.1})-(\ref{5.3}) one gets
$$\|\prod_n \psi_{\lambda_n}(z)-1\|_{K,\infty} \leq \prod_{n=1}^\infty (1+\frac{\delta}{2^{n+1}} )-1 <\delta.$$
We immediately get that for any Blaschke product $ \phi$, whenever
(\ref{5.3}) holds,
\begin{equation}\| \phi \prod_n \psi_{\lambda_n} - \phi\|_{K,\infty} \leq \delta .\end{equation}
Notice that (5.4) also holds if   $ \{ \lambda_n \}$ is replace with
any subsequence.

Now we are ready to give the construction of $B$.

\noindent \textbf{Proof of Theorem 5.3.} Assume that $\lim\limits_{n\to \infty} |z_n|=1$. First we show that
  $\{z_n\}$ contains  a thin Blaschke subsequence. To see this,
let $d$ denote the hyperbolic metric over the unit disk, and we have $\lim\limits_{n\to \infty} d(z_k,z_n)=1$ for each $k$. 
  Pick two natural numbers $m_1$ and $m_2$ such that  $d(z_{m_1},z_{m_2})>1-\frac{1}{2^2}, $  and put $ w_j=z_{m_j}(j=1,2)$. In general, assume that  we have
  $n$ points:
$$  w_1=z_{m_1},\cdots, w_n=z_{m_n},$$
and $$\prod_{1\leq j<k} d(w_j, w_k)>1-\frac{1}{(k+1)^2},
k=1,2,\cdots, n.$$ Since $\lim\limits_{k\to\infty}\prod_{1\leq j
\leq n} d(w_j, z_k)=1, $ then there is an integer $m_{n+1}$ such that
$m_{n+1}>m_n$
$$\prod_{1\leq j \leq n} d(w_j, z_{m_{n+1} })>1- \frac{1}{(n+2)^2}.$$ Put $w_{n+1}=z_{m_{n+1}}$, and the induction is finished.
That is, we have a subsequence $\{w_n\}$ satisfying
$$\prod_{1\leq j<k} d(w_j, w_k)>1-\frac{1}{(k+1)^2}, k=1,2,\cdots. $$
Now we have
\begin{eqnarray*} \prod_{j;j\neq k} d(w_j, w_k)&=& \prod_{j< k} d(w_j, w_k)\prod_{j> k} d(w_j, w_k) \\
 & \geq &\big(1-\frac{1}{(k+1)^2}\big)\prod_{j> k} \big(1-\frac{1}{(j+1)^2}\big)\to 1, k\to \infty.\end{eqnarray*}
This immediately shows that $\{w_n\}$ is a thin Blaschke sequence,
as desired.

Now we may assume that $\{z_n\}$ is  itself  a thin Blaschke
sequence.
We will find a   subsequence  of $\{z_n\}$ whose Blaschke product
$B$ satisfies the conditions (i) and (ii) in Theorem \ref{41}. For
this, we will construct a sequence $B_m$ of finite Blaschke products
with $\deg B_m=m+1$ and $
 B $ equals the limit of $B_m$.

 Write $$  \varphi = \varphi_{z_1} \varphi_{z_2}, $$
 where $$\varphi_{\lambda}(z)=\frac{ \lambda-z}{1-\overline{\lambda}z}.$$
 By Theorem \ref{47}, $\varphi$ has a  unique critical point,  denoted by $w_1$. Namely, $\varphi'(w_1)=0$.
Rewrite $B_1=\varphi$, and one can give  two closed  disks
$\widetilde{K_1}$,
  $K_1$ in $\mathbb{D}$, and a constant $\delta_1>0$ satisfying the following:
   \begin{itemize}
\item[(1)] $\widetilde{K_1}$ is a closed disk containing $\{z:|z|< \frac{1}{2}\}$ and $w_1$.
\item[(2)]  $\widetilde{K_1} \subseteq \stackrel{\circ} {K_1}$,  $\delta_1 $ is enough small    such that for any holomorphic function $h$ with
  $\|h-B_1\|_{K_1,\infty}<\delta_1$,   $Z(h')\cap \widetilde{K_1}$ contains exactly one point.
  \item[(3)] Taking  $\delta=\delta_1$ and $K=K_1$ in (\ref{5.3}), there is a subsequence $\{z_{n,1}\}_{n=1}^\infty$ of $\{z_n\}_{n=1}^\infty$ satisfying (\ref{5.3}).
 \end{itemize}
 By induction, suppose that we have
already constructed a finite Blaschke product $B_m$ with $\deg
B_m=m+1$ and satisfying the following:
\begin{itemize}
\item[(i)]  $B_m|_{Z (B_m ')}$ is injective;
\item[(ii)] $\widetilde{K_m}$ is a closed disk containing $\{z:|z|<1-\frac{1}{m+1}\}$ and $Z(B_m')$.
\item[(iii)] there is a $\delta_m>0$ and two closed disks $\widetilde{K_m}$ and $K_m$ ($\widetilde{K_m} \subseteq \stackrel{\circ} {K_m}$) in $\mathbb{D}$ such that for any holomorphic function $h$ with
  $\|h-B_m\|_{K_m,\infty}<\delta_m$, the restriction of $h $ on $Z(h')\cap \widetilde{K_m}$ is injective.
  \item[(iv)] Taking  $\delta=\delta_m$ and $K=K_m$ in (\ref{5.3}), there is a subsequence $\{z_{n,m}\}_{n=1}^\infty$ of $\{z_{n,m-1}\}_{n=1}^\infty$ satisfying (\ref{5.3}).
 \end{itemize}
By Lemma \ref{51} there is a  $ \lambda\in  \{z_{n,m}\}_{n=1}^\infty$ with $1-|\lambda|$
enough small such that the restriction of  $B_m \varphi_\lambda$  on
$Z\big((B_m\varphi_\lambda)' \big)$ is injective. Put \linebreak $B_{m+1}=B_m
\psi_\lambda$; that is, $B_{m+1}$ is a constant multiple of $B_m
\varphi_\lambda$. Now set $$r=\max \{\frac{1+r(K_m)}{2},1-\frac{1}{m+2}\}$$
and
 $r'=\frac{1+r}{2}$. Set
 $$\widetilde{K_{m+1}}=\{z:|z|\leq r\} \quad \mathrm{and} \quad  K_{m+1} =\{z:|z|\leq r'\}.$$
 Clearly, $\widetilde{K_{m+1}}\subseteq  K_{m+1}$.
By Lemma \ref{52} we can find a $\delta_{m+1}>0$ such that
 the similar version of (iii) holds for $m+1$.
Taking $$\delta=\delta_{m+1} \ \ \mathrm{and}\ \  K=K_{m+1}$$ in  (\ref{5.3}), there
is a subsequence $\{z_{n,m+1}\}_{n=1}^\infty$ of
$\{z_{n,m}\}_{n=1}^\infty$ satisfying (\ref{5.3}). Considering that
  both   $z_{1,m+1}$ and $\lambda$ lie in $ \{z_{n,m}\}_{n=1}^\infty$, it is required  that $z_{1,m+1}$ is behind $\lambda$.
 Now the induction on $m$ is complete. 

 By our induction, it is clear that  $B_m$   converges to a thin Blaschke product $B$. For each fixed $n$, there is
 a  Blaschke product $\widetilde{B_m}$ such that
 $B=B_m \widetilde{B_m}.$ Let $\{\lambda_k\}$ be the zero   sequence of  $\widetilde{B_m}$
  (of course, $\{\lambda_k\}$ depends on $m$).
By our construction, $\{\lambda_k\}$ is a subsequence of $
\{z_{n,m}\}_{n=1}^\infty $, and then by (iv)  (5.4) holds:
$$ 1-|\lambda_k|\leq \frac{\delta_m (1-r(K_m))}{2^{k+2}} .$$
Since $\widetilde{B_m}=\prod_k \psi_{\lambda_k}$,
then  by (5.4)
$$\|B_m  \widetilde{B_m} -B_m\|_{K_m,\infty} \leq \delta_m .$$
That is, $$\|B -B_m\|_{ K_m ,\infty} \leq \delta_m.$$ Thus by (iii),
$B|_{Z(B')\cap \widetilde{K_m}}$ is injective. Since
$\widetilde{K_m}$ is an increasing sequence of closed domains whose
union is $\mathbb{D}$,
 $B|_{Z(B')}$ is injective.
  Applying Lemma \ref{52} shows that $B'|_{\widetilde{K_m}}$ has only simple zeros for each $m$, and hence $B'$ has only simple zeros.
  Thus, $B$ satisfies the assumptions of Theorem \ref{41}, with $F$ being empty.
  Therefore, the construction of $B$ is complete and the proof is finished.
  $\hfill \square$

 \section{Geometric characterization  for when $\mathcal{V}^*(B)$ is nontrivial}
~~~~ In this section, we will give a geometric characterization for
when $\mathcal{V}^*(B)$ is nontrivial, as follows. Its proof is delayed.
\begin{thm} Suppose that $B $ is  a thin Blaschke product satisfying \linebreak $\overline{Z(B)}\nsupseteq \mathbb{T}$.
Then $\mathcal{V}^*(B)$ is nontrivial if and only if    the
Riemann surface $S_B$ has a nontrivial component with  \label{62}
finite multiplicity.
 \end{thm}

We follow the notations and terms in Section 3. It has been  shown
that for any local inverse $\rho$ of $B$,
 $ G[\rho] $ is a Riemann surface, and $\pi_\rho: G[\rho]\to E$ is a covering map. The number   $\sharp (\pi_\rho )^{-1}(z)=\sharp [\rho](z)$,
 is called  \emph{the number of sheets for} $G[\rho]$, also denoted by  $\sharp [\rho]$. Clearly,  the
inverse map $\rho^{-1}$  of $\rho$ is also a local inverse   of $B$,
and we rewrite $\rho^-$ for $\rho^{-1}$. If both $\sharp [\rho]$ and
$\sharp [\rho^-]$ are finite, then we say that the component
$G[\rho]$ of $S_B$ has \emph{ bi-finite multiplicities}.

For any subset $F$ of $E\times E$, we write $F^-=\{(w,z) : (z,w)\in
F \}.$ Thus,
  $$G^-[\rho]=\{(w,z) : (z,w)\in G[\rho] \}.$$
We have a geometric characterization for $G[\rho^-]$  as follows.
\begin{prop}For each local inverse $\rho$, we have $ G[\rho^-]=G^-[\rho]$. \label{61}
\end{prop}

\begin{proof}
Suppose that $z_0\in D(\rho)$. Clearly, $\rho^-\big(
\rho(z_0)\big)=z_0.$ For any curve $\gamma$ that connects $z_0$ with
$z(z\in E)$, suppose that $\widetilde{\rho}$ is the analytic continuation of $\rho$ along $\gamma$, and then
$\widetilde{\rho}(\gamma)$ is a curve connecting $\rho(z_0)$ and $\widetilde{\rho}(z)$. Let $\widetilde{\rho^-}$ be the analytic
continuation of $\rho^-$
 along $\widetilde{\rho}(\gamma)$.
Since $\rho^-\big( \rho(z_0)\big)=z_0,$   then  $$\widetilde{\rho^-} (\widetilde{\rho}  )=id$$
 holds on the curve $\gamma$. In particular,
 $\widetilde{\rho^-}\big( \widetilde{\rho}(z)\big)=z$, forcing $$( \widetilde{\rho}(z),z )\in G[\widetilde{\rho^-}]=G[\rho^-]  .$$
 Therefore,
$$G^-[\rho]\subseteq G[\rho^-].$$
Therefore, $$\{G^-[\rho]\}^- \subseteq G^-[\rho^-]  \subseteq
G[\{\rho^-\}^-],$$ Clearly,
$$ G[\rho]  \subseteq G^-[\rho^-]  \subseteq G[ \rho ].$$
which immediately gives that $G^-[\rho^-]=G[\rho]$, and hence
$G[\rho^-]=G^-[\rho]$.
\end{proof}
 Since  $\pi_{\rho^-}:G[\rho^-]\to E $ is a covering map, then by Proposition \ref{61}    $\pi_2:G[\rho]\to E, (z,w) \to w$ is also a covering map, and
 the multiplicity of $\pi_{\rho^-} $ (=$\sharp [\rho^-]$) equals that of $\pi_2|_{G[\rho ]} $.
Notice that   $S_B$ has a nontrivial component with bi-finite
multiplicities if and only if there exist a nontrivial local inverse
$\rho$ of $B$,
     such that both  covering maps $\pi_\rho:G[\rho]\to E$ and $\pi_2:G[\rho]\to E$ have finite multiplicities.

Before continuing, we need some preparations. If   $ \sharp [\rho]
=n<\infty$, then $\pi_\rho :G[\rho ]\to E $ is a covering map of finite
multiplicity. Then for each $z\in E$, there is a small disk
$\Delta$ containing $z$  and $n$ local inverse $\rho_i(1\leq i \leq
n)$ defined on $\Delta$
 such that   $$\pi_\rho^{-1}(\Delta)=\bigsqcup_{i=1}^n \{(w,\rho_i(w)):w\in \Delta\}.$$
In this case, we set $$  \sum_{ \sigma\in [\rho]} h\circ \sigma (w)
\,  \sigma' (w)\triangleq
 \sum_{i=1}^n h\circ \rho_i (w) \,  \rho_i ' (w), w\in \Delta,$$ where $h$ is any function over $E$ or $\mathbb{D}$.
Then   define a linear map $\mathcal{E}_{[\rho]} $ by
 $$\mathcal{E}_{[\rho]}h(z) \equiv \sum_{ \sigma\in [\rho]} h\circ \sigma (z) \,  \sigma' (z) ,\, z\in E.$$
One can check that the map $\mathcal{E}_{[\rho]} $ is well defined,
i.e. the value $\mathcal{E}_{[\rho]}h(z)$
 does not depend  on the choice of disks $\Delta$ containing $z$.
The functions $h$ over $E$ or $\mathbb{D}$ are not necessary
holomorphic. Sometimes, we write
$$ \mathcal{E}_{[\rho]}h =\sum_{\rho_i\in [\rho]}h\circ \rho_{i } \rho_{i }'  $$
to emphasize that this $\mathcal{E}_{[\rho]} $ maps   holomorphic
functions to holomorphic functions.
 As a special case, the trivial component $[z]$ gives the identity operator $I$, i.e. $ \mathcal{E}_{[z]} =I$.
 We mention that those operators $\mathcal{E}_{[\rho]} $  first appeared in \cite{DSZ}.
Similarly, we  define $$\mathcal{E}_{[| \rho|]}h(z)=\sum_{\rho_i\in
[\rho]}h\circ \rho_{i }(z)| \rho_{i }'(z) |. $$

It is well know that  a square-summable holomorphic function on a
punctured disk (a disk minus one internal point) can always be
extended analytically to the whole disk. Also notice   that $
\mathbb{D}-E$ is discrete in $ \mathbb{D},$  and thus each function
$f $ in $ L^2_a( E)$ can be   extended analytically to the unit disk
$\mathbb{D}$; that is, $f$ can be regarded as a member in $L^2_a(
\mathbb{D})$. In this sense, we write
$$ L^2_a( E)=L^2_a( \mathbb{D}).$$

 As above mentioned,  $\mathcal{E}_{[ \rho]}$ maps   holomorphic functions  to
holomorphic functions. Soon, one will see that
 $\mathcal{E}_{[ \rho]}$ defines a linear map from $ L^2_a( \mathbb{D})$ to $ L^2_a( E)$. Moreover, the map
  $\mathcal{E}_{[ \rho]}:  L^2_a( \mathbb{D})\to  L^2_a( \mathbb{D}) $ is a bounded operator, which is given by the following.
\begin{lem}If $G[\rho]$ is a component of $S_B$  with  bi-finite multiplicities, then both \label{63}
$\mathcal{E}_{[\rho]} $  and $\mathcal{E}_{[\rho^-]} $  are in
$\mathcal{V}^*(B)$, and $\mathcal{E}_{[\rho]}^*=\mathcal{E}_{[\rho^-]}$.
\end{lem}
\begin{proof}
Assume that  $G[\rho]$ is a component of $S_B$  with  bi-finite
multiplicities. First we will show that $\mathcal{E}_{[\rho]} $  is
bounded. To see this, let $h\in C_c(E)$; that is, $h$ is a compactly
supported continuous function  over $E$.
 For each open set $U$ that admits a complete local inverse $\{\rho_i\}_{i=0}^\infty$,
    we have \begin{eqnarray}
    \int_U |\mathcal{E}_{[\rho ]}h(z)|^2dA(z) &\leq &   C \int_U  \sum_{\rho_i\in [\rho]}
  | h\circ \rho_{i}(z) |^2 |\rho_{i}'(z)|^2 dA(z) \nonumber \\
   &=& C\sum_{\rho_i\in [\rho]} \int_{\rho_i(U)}  | h(z)|^2   dA(z) ,            \label{6.1}
\end{eqnarray}
where $C$ is a numerical constant that depends only on $n $.

 Since $E$ is an open set, it can be represented as the union of countable open disks, each admitting a complete local inverse.
The reasoning is as follows. For each $z\in E$, there is an
$r=r(z)>0$ such that the disk $O(z,r(z))$  admits a complete local
inverse. We  can choose a   rational number $r'=r'(z)$  with $\frac{r}{3}<r'<
\frac{r}{2}$ and a point $z'$ in $ E$ with rational coordinates such that
$|z-z'|<\frac{r}{3}$. Then it is easy to verify that $$z\in O(z', r')
\subseteq O(z,r(z)).$$ All such disks $O(z',r') $ consist of a countable
covering of $E$. Also, each member $O(z',r') $ admits a complete
local inverse.

Now assume that $\{U_j\}$ is   a covering of  $E$ and each $U_j$
admits a complete local inverse   $\{\rho_i\}_{i=0}^\infty$. We
emphasize that the family $\{\rho_i\}_{i=0}^\infty$ is fixed once
their common domain $U_j$ of definition is fixed. Set $$E_1=U_1
\quad \mathrm{and } \quad E_j=U_j-\bigcup_{1\leq i<j }U_i \, (j\geq
2).$$ By the same reasoning    as (\ref{6.1}), we also have
$$ \int_{E_j} \Big|\mathcal{E}_{[|\rho | ]}|h|(z)\Big|^2dA(z) \leq  C\sum_{\rho_i\in [\rho]} \int_{\rho_i(E_j)}  | h(z)|^2   dA(z)
 \equiv C\sum_{\rho_i\in [\rho]}\mu (\rho_i(E_j)), $$
where  $\mu$ is a measure defined by $\mu (F)=\int_F |h(z)|^2dA(z)$
for each  Lebesgue measurable subset $F$ of $\mathbb{D}$. Since
$E=\cup U_j=\sqcup E_j$, then
$$\int_{E } \Big|\mathcal{E}_{[|\rho| ]}|h|(z)\Big|^2dA(z) \leq C\sum_{j=1}^\infty \sum_{\rho_i\in [\rho]} \mu (\rho_i(E_j)).$$
Write $m=\sharp [\rho^-]<\infty.$ We need a  fact in measure theory: for measurable sets $F$ and $F_j$, if $\cup_{j=1}^k F_j
\subseteq F$, and for each point $z\in F$, there are at most $m$
sets  $F_j$ such that $z\in F_j$, then  $$\sum_{j=1}^k \chi_{F_j}\leq m
\chi_F,$$ and hence $\sum_{j=1}^k \mu (F_j)\leq m
\mu (F).$ Below, we will apply this result. By our choice of $E_j $
($E_j\subseteq U_j$), it follows from Proposition \ref{61} that   for any
$w\in E$, there are at most $m$ sets $E_j$  such that
 $w\in \rho_i(E_j)$ for some $\rho_i\in [\rho].$  Therefore,
 $ \sum_{\rho_i\in [\rho]}\sum_{j=1}^k \mu(\rho_i(E_j))    \leq m \mu(E).$ Letting $k$ tends to infinity shows that
 $$  \sum_{\rho_i\in [\rho]}\sum_{j=1}^\infty \mu(\rho_i(E_j))  \leq m \mu(E);$$
and thus
 $$\int_{\mathbb{D}} \Big|\mathcal{E}_{[|\rho |]}|h|(z)\Big|^2dA(z)\leq mC \|h\|^2.$$
 By a limit argument, it is easy to see that the above identity also holds for any $h\in C(\mathbb{D})$,
  which immediately gives that
 $$\int_{\mathbb{D}} |\mathcal{E}_{[ \rho  ]}f(z)|^2dA(z)\leq \int_{\mathbb{D}} \Big|\mathcal{E}_{[|\rho |]}|f|(z)\Big|^2dA(z)\leq mC \|f\|^2, f\in L^2_a(\mathbb{D}).$$
 That is, $\mathcal{E}_{[ \rho  ]}:  L^2_a(\mathbb{D})\to  L^2_a(E)$ is a bounded operator.
 By the comments above  Lemma \ref{63}, $\mathcal{E}_{[  \rho ]}$ can be regarded as a bounded operator from
 $ L^2_a( \mathbb{D})$ to $ L^2_a( \mathbb{D}) $.
Similarly, $\mathcal{E}_{[  \rho ^-]}$ is a bounded operator from
 $ L^2_a( \mathbb{D})$ to $ L^2_a( \mathbb{D}) $.

Below we will show that  $
\mathcal{E}_{[\rho]}^*=\mathcal{E}_{[\rho^-]}.$
Now let $U$ be a fixed open subset of $E$ admitting a complete local
inverse,
    $h$ be a bounded (measurable) function whose support is contained in $\rho_{i_0}(U)$ for some $i_0$, and $g\in C_c(E).$
 Since  $ m=\sharp[\rho^-] $, then we may write
  \begin{equation}[\rho^-]\rho_{i_0}(U) = \bigsqcup_{j=1}^m \check{\rho_j}(U). \label{6.2}
  \end{equation}
  That is, on $\rho_{i_0}(U)$ there are $m$ members $\rho_{k_j}(1\leq j \leq m)$ in $[\rho^-]$,
  and there are $m$ integers $k_j'(1\leq j \leq m)$ such that
   $$\rho_{k_j}\circ \rho_{i_0}(z)= \rho_{k_j'}(z),\, z\in U.$$
   Rewrite $\check{\rho_j}$ for $\rho_{k_j'}(1\leq j \leq m)$, and then we have (\ref{6.2}).

 Then for each $j(1\leq j\leq m), $ there is a   local inverse $\sigma_j  : \check{\rho_j}(U) \to \rho_{i_0}(U) $ in $   [\rho]$ and $\sigma_j$
 is onto.  Since $ \sigma_j\in [\rho] $ and $h$ is supported on $\rho_{i_0}(U),$ then
  \begin{eqnarray*}\int_{\check{\rho_{j}}(U)} &\sum_{\rho_k \in [\rho]} & h\circ \rho_{k}(z) \rho_{k}'(z)\overline{ g(z)}dA(z)
    =      \int_{\check{\rho_{j}}(U)} h\circ \sigma_j(z) \sigma_j'(z)\overline{ g(z)}dA(z)\\
   &=&   \int_{\rho_{i_0}(U)} h\circ \sigma_j(\sigma_j^-(w)) \sigma_j'(\sigma_j^-(w)) | \sigma_j^{-'}(w)|^2  \overline{g(\sigma_j^-(w))}dA(w)  \\
&=&   \int_{\rho_{i_0}(U)}  h(w)   \overline{  g(\sigma_j^-(w))
\sigma_j^{-'}(w)}  dA(w).
\end{eqnarray*}
  That is,
$$\int_{\check{\rho_{j}}(U)} \sum_{\rho_k \in [\rho]} h\circ \rho_{k}(z) \rho_{k}'(z)\overline{ g(z)}dA(z)
    =  \int_{\rho_{i_0}(U)}  h(w)   \overline{  g(\sigma_j^-(w)) \sigma_j^{-'}(w)}  dA(w).$$
Since $h$ is  supported   in $\rho_{i_0}(U)$, then by (\ref{6.2})
one can verify that
  $$ \mathcal{E}_{[\rho]}h=
\sum_{\rho_i\in [\rho]}\sum_{j=1}^m \chi_{\check{\rho_j}(U)}h\circ
\rho_i \rho_i',$$ and hence
\begin{eqnarray*}
 \langle    \mathcal{E}_{[\rho]} h, g \rangle &=&
 \sum_{j=1}^m  \int_{\check{\rho_{j}}(U)}\sum_{\rho_i\in [\rho]}  h\circ \rho_{i}(z) \rho_{i}'(z)\overline{ g(z)}dA(z)  \\
&=&    \sum_{j=1}^m \int_{\rho_{i_0}(U)}  h(w)   \overline{  g(\sigma_j^-(w)) \sigma_j^{-'}(w)}  dA(w). \\
&=& \int_{\mathbb{D}}  h(w)  \chi_{\rho_{i_0}(U)} \overline{\sum_{\sigma \in [\rho^-]} g(\sigma(w))\sigma'(w) }  dA(w) \\
&=&   \langle   h,  \mathcal{E}_{[\rho^-]}g \rangle   .
\end{eqnarray*}
That is, $$ \langle    \mathcal{E}_{[\rho]} h, g \rangle=\langle   h,  \mathcal{E}_{[\rho^-]}g \rangle   . $$
Since each compact subset of $E$ is always contained in the
union of finite open sets like $\rho_{i_0}(U)$, the above identity
also holds for any $h\in C_c(E).$ Then by a limit argument, the
above holds for any $g,h\in L^2_a(\mathbb{D})$. Therefore, $
\mathcal{E}_{[\rho]}^*=\mathcal{E}_{[\rho^-]}.$ Moreover, since both
$\mathcal{E}_{[\rho^-]}$ and $\mathcal{E}_{[\rho]}$  commute with
$M_B$, then $\mathcal{E}_{[\rho]}\in \mathcal{V}^*(B)$. The proof is complete.
\end{proof}

In Theorem \ref{32}, we get a global  representation for unitary
operators in $\mathcal{V}^*(B)$, as follows:
$$ Sh(z)=\sum_{k=0}^\infty c_k  h\circ \rho_k(z) \rho_k'(z),
 \ h\in L^2_a( \mathbb{D}),\, z\in \mathbb{D}-L . $$
 By Lemma \ref{35}, we may rewrite $S$ as
 \begin{equation}
 Sh(z)=\sum_{ [\rho]}\sum_{\sigma \in [\rho]} c_\rho  h\circ \sigma(z) \sigma'(z) \equiv \sum_{ [\rho]}  c_\rho \mathcal{E}_{[\rho]} h(z)
 ,\ h\in L^2_a( \mathbb{D}) ,\, z\in E  . \label{6.3}
 \end{equation}

To give the proof of Theorem \ref{62}, we also need to establish the following, whose proof is delayed.
\begin{lem} Given a thin Blaschke product $B$, suppose that $S $ is an operator in  $\mathcal{V}^*(B)$ with the   form (\ref{6.3}).
If    $\overline{Z(B)}\nsupseteq \mathbb{T}$, then  $$S^*h(z)=\sum_{ [\rho]} \overline{c_\rho}  \mathcal{E}_{[\rho^-]}
h(z),  \ h\in L^2_a( \mathbb{D}),\, z\in E  .$$  \label{64}
Furthermore, for each $\rho$ such that $c_{\rho}\neq 0$, we have $\sharp
[\rho]<\infty $ and $\sharp [\rho^-]<\infty$.
 \end{lem}
\noindent  In Section 7,    we will show that there does exist a thin Blaschke product $B$
  for which  the condition $\overline{Z(B)}\nsupseteq \mathbb{T}$ fails, see Example 7.1.

Applying Lemma \ref{64} and the methods in \cite[Section 2]{GH2}, we have
the following. For the case of finite Blaschke products, see
\cite[Theorem 7.6]{DSZ}.
  \begin{cor} If  $B$ is a thin Blaschke product with   $\overline{Z(B)}\nsupseteq \mathbb{T}$, \label{65}
  then the von Neumann algebra $\mathcal{V^*}(B)$ is generated by all $\mathcal{E}_{[\rho]}$, where $G[\rho]$ have finite multiplicity.
 \end{cor}
 \begin{proof} Suppose that
 $B$ is a thin   Blaschke product and $\overline{Z(B)}\nsupseteq \mathbb{T}$. By Lemma \ref{63},
 we have $$span\, \{ \mathcal{E}_{[\rho]}:\sharp[\rho]<\infty \} \subseteq \mathcal{V}^*(B).$$
  By von Neumann bicommutant theorem, the proof of Corollary \ref{65} is reduced to show that the
commutant of $span\, \{ \mathcal{E}_{[\rho]}:\sharp[\rho]<\infty \} $ equals that of
$\mathcal{V}^*(B)$. For this, for a given operator $A$ commuting
with all $\mathcal{E}_{[\rho]}$, we must show that
 $A$ commutes with each operator  in $\mathcal{V}^*(B)$.

For this, we first present a claim. For any $S  \in \mathcal{V}^*(B )  $,
write
\begin{equation}S=\sum c_\rho \mathcal{E}_{[\rho]} .  \label{thin}\end{equation}
By Lemma \ref{64}, for each $c_\rho\neq 0,$ we have $\sharp [\rho]<\infty  $ and $\sharp [\rho^-]<\infty$
Let $S_m$  denote the corresponding partial sum of  the expanding series (\ref{thin}) for $S$.
 We claim that for all $\lambda \in E$,
\begin{equation}\lim_{ m\to \infty}\|S_m K_\lambda\to S K_\lambda \|=0 . \label{534}\end{equation}

If we have shown this claim, then we can give the proof of Corollary \ref{65} as follows.
For any $S \in\mathcal{V}^*(B)$,   define
 $S_m$ as  above. Notice that $A$ commutes with $S_m$, and hence
 \begin{equation}\langle A S_m K_z,K_z\rangle=\langle AK_z, S_m^*
  K_z\rangle,\,z\in E. \label{eq}\end{equation}
By the claim and Lemma \ref{64}, we have
$$\lim_{m  \to \infty} \|S_m^*K_z-S^*K_z\|=\lim_{m\to \infty} \|S_mK_z-SK_z\|=0, \,  z\in E. $$
Take limits in (\ref{eq}) and we get
 $$\langle AS K_z,K_z\rangle=\langle AK_z, S^* K_z\rangle=\langle SAK_z,  K_z\rangle.$$
By the continuity of both sides in variable $z$, the above holds for every $z\in \mathbb{D}$.
The property of Berezin transformation yields that $AS=SA$, as desired.

Now it remains to show (\ref{534}). The reasoning is as follows.
For each $\rho$ satisfying  $\sharp[\rho]<\infty$  and  $\sharp [\rho^-]<\infty$.
  Lemma  \ref{63} shows that $\mathcal{E}_{[\rho]}^*=\mathcal{E}_{[\rho^-]}$,
Then by $$\langle \mathcal{E}_{[\rho]} K_\lambda , f \rangle =\langle K_\lambda, \mathcal{E}_{[\rho]}^* f \rangle, f\in L_a^2(\mathbb{D}),$$
we get
\begin{equation} \mathcal{E}_{[\rho]} K_\lambda =\sum_{\sigma \in [\rho^-]}\overline{ \sigma (\lambda)'}K_{\sigma (\lambda)}.\label{535}\end{equation}
Now   rewrite $$S=
\sum d_\rho'  \mathcal{E}_{[\rho^-]} \equiv  \sum d_\rho'  \mathcal{E}_{[\rho]}^*.$$
On a neighborhood $V$ of $\lambda ,$ the local inverse of $B$ can be written as $\{\rho_k\}$.
In this case, we have $$ Sh(z)=\sum_{k=0}^\infty d_k' h\circ \rho_k (z) \rho_k' (z), z\in V,$$
with $d_k'=d_{\rho_k}'.$
This, combined with (\ref{535}), yields that
$$S K_\lambda (z) =\sum_{k=0}^\infty d_k' \overline{\rho_k' (z)} K_{\rho_k(\lambda)}(z) , z\in V.$$
Below, we will show that $\sum_{k=0}^\infty d_k'  \overline{\rho_k' (z)} K_{\rho_\lambda}$ converges in norm,
which immediately shows that $\lim S_m K_\lambda=SK_\lambda$, as desired.

 In fact,   on $V$ we have
  $B\circ \rho_k  =B$ for each $k=0,1,\cdots$,  which gives that
\begin{equation}|B'(\rho_k(\lambda)) | |\rho_k'(\lambda)| =|B'(\lambda)|.\label{536}\end{equation}
Write $ a=B(\lambda)$,   by Proposition \ref{24} $\varphi_a(B)$ is a thin Blaschke product whose zero sequence is $\{\rho_k(\lambda)\}$.
Then there is a numerical constant $\delta>0$ such that $$\prod_{j,j\neq k}d(\rho_j(\lambda),\rho_k(\lambda) )\geq \delta, $$
and then by computation we get
$$|(\varphi_a\circ B)'(\rho_k(\lambda)) |=  \frac{1}{1-|\rho_k(\lambda)| ^2 }\prod_{j,j\neq k}d(\rho_j(\lambda),\rho_k(\lambda) )\geq
 \delta\frac{1}{1-|\rho_k(\lambda)| ^2 }.$$
Since $ B (\rho_k(\lambda))=B ( \lambda )=a$, the above immediately shows that
 $$\frac{1}{1-|\rho_k(\lambda)| ^2 }\geq  |\varphi_a'(a)  | |B '(\rho_k(\lambda))| \geq
 \delta\frac{1}{1-|\rho_k(\lambda)| ^2 },$$
 and hence    $$ \frac{1}{|\varphi_a'(a)|}\frac{1}{1-|\rho_k(\lambda)| ^2 }\geq  |B '(\rho_k(\lambda))| \geq
  \frac{\delta}{|\varphi_a'(a)|}\frac{1}{1-|\rho_k(\lambda)| ^2 }.$$
  Then by (\ref{536}), there is a positive constant $C$ such that
  \begin{equation} \frac{1}{C}(1-|\rho_k(\lambda)|^2 )\leq  |\rho_k'(\lambda)|\leq  C(1-|\rho_k(\lambda)|^2).\label{537}\end{equation}
On the other hand,   the uniformly separated sequence  $\{\rho_k(\lambda)\}$ is interpolating
for $L_a^2(\mathbb{D})$; that is,  the equations $f(z_k)=w_k$ for $k=1,2,\cdots$ has a solution
$f$ in $L_a^2(\mathbb{D})$ whenever
$$\sum_{k=1}^\infty (1-|z_k|^2)^2|w_k|^2 <\infty.$$ Then applying the proof of
\cite[Lemma A]{Cow2} shows that
$$A:  h \to \{ h (\rho_k(\lambda))(1-|\rho_k(\lambda)|^2)\}   $$
is a bounded invertible operator from $ (\varphi_a(B)L_a^2(\mathbb{D}))^\perp$ onto $l^2$.
Denote by $\{e_k\}_{k=0}^\infty$ the standard orthogonal basis of $l^2$, and
 then $$A^*e_{k}=(1-|\rho_k(\lambda)|^2) K_{\rho_k(\lambda)},k=0,1,\cdots.$$
 Then by (\ref{537}),
$$\widetilde{A} e_{k}= \rho_k'(\lambda)K_{\rho_k(\lambda)}$$
naturally induces a (linear) bounded operator $\widetilde{A}$. This immediately shows that
$$\sum_{k=0}^\infty d_k'  \overline{\rho_k' (z)} K_{\rho_k(\lambda)}$$ converges in norm, and the sum does not depend on the order of $k$.
Since $\{S_m K_\lambda\}$ is a subsequence of the partial sum of this series,
$\{S_m K_\lambda\}$ converges to $S  K_{\rho_\lambda}$ in norm.
 \end{proof}
 \begin{rem} In fact, with more efforts one can show that in Theorem \ref{62} and Corollary \ref{65},
 the condition that $\overline{Z(B)}\nsupseteq \mathbb{T}$ can be dropped.
 \end{rem}
Now with Lemma \ref{64}, we can give the proof of Lemma \ref{62}.

\noindent \textbf{Proof of Theorem \ref{62}.} Suppose that $B $ is  a
thin Blaschke product and $\overline{Z(B)}\nsupseteq \mathbb{T}$.
By Lemma \ref{63}, it suffices to show that $\mathcal{V}^*(B)$ is nontrivial if and only if
  there is a nontrivial local inverse  $\rho$ such that   $\sharp [\rho]<\infty $ and $\sharp [\rho^-]<\infty$.

  If $\mathcal{V}^*(B)$ is nontrivial, then there is an operator $S\in \mathcal{V}^*(B)$ different from a constant multiple  of $I$.
  By Theorem \ref{32}, we may assume that $S$ has the following form: $$
 Sh(z)=\sum_{k=0}^\infty c_k  h\circ \rho_k(z) \rho_k'(z),
 \ h\in L^2_a( \mathbb{D}),\, z\in \mathbb{D}-L . 
$$
Then  there is a nontrivial local inverse $\rho_k$ such that
$c_k\neq 0$. By Lemma \ref{64}, both $\sharp [\rho_k]$ and  $\sharp
[\rho_k^-]$  are finite, as desired.

The inverse direction follows directly from Lemma 6.3.  Therefore, the proof is
complete. $\hfill \square$
\vskip2mm
For  completeness,  we must prove Lemma \ref{64}. For this, we need some
preparations. \vskip1.5mm \noindent\textbf{Claim:} If
$\overline{Z(B)}\nsupseteq \mathbb{T}$, then  for any $0<\delta <1,$
$\overline{B^{-1}(\Delta(0,\delta))}\nsupseteq \mathbb{T}.$

\vskip1.5mm The reasoning is as follows. By (\ref{2.2}), there is
an $\varepsilon\in(0,1)$ such that $$
 B^{-1}(\Delta(0,\delta))\subseteq \bigcup_n \Delta(z_n,\varepsilon),
$$  where $Z(B)=\{z_n\}.$
 Therefore, it suffices to show that $\overline{\bigcup_n \Delta(z_n,\varepsilon)} \nsupseteq \mathbb{T}.$
For this, assume conversely that $\overline{\bigcup_n
\Delta(z_n,\varepsilon)}  \supseteq  \mathbb{T}$. Since
$\overline{Z(B)}\nsupseteq \mathbb{T}$,
 there is a $\zeta\in \mathbb{T} $ such that $ \zeta \not\in \overline{Z(B)}.$
Since $\overline{\bigcup_n \Delta(z_n,\varepsilon)}  \supseteq
\mathbb{T}$, there is a sequence $\{w_k\}$ in $\mathbb{D}$ and a
subsequence $\{z_{n_k}\}$ of $\{z_n\}$ such that
 \begin{equation}d( z_{n_k},w_k)<\varepsilon \label{6.4}\end{equation}
and $$\lim_{k\to \infty}w_k =\zeta.$$
 Since $\zeta \not\in
\overline{Z(B)}   $, we have $$d(z ,\zeta)=1, z\in \overline{Z(B)}
.$$ By the continuity of $d(z,w)$, there is an enough small closed
disk $ \overline{D}  $ centered at $\zeta$ such that
$$d(z ,w)\geq  \frac{1+\varepsilon}{2}, z\in \overline{Z(B)} , w\in Vw_k.$$
For enough large $k$, $w_k \in w_k$ and then $$d(z_{n_k},w_k
)\geq \frac{1+\varepsilon}{2},$$ which is a contradiction to
(\ref{6.4}). Therefore, $\overline{B^{-1}(\Delta(0,\delta))}\nsupseteq \mathbb{T}.$ The
proof of the claim is complete. $\hfill \square$

\vskip2mm By the claim, for each $z\in E$  there is always an open
neighborhood  $U$ ($U\subseteq E$) of $z$ that admits a complete
local inverse, such that
  $$\overline{B^{-1}\circ B(U)}  \nsupseteq  \mathbb{T}.$$ Then we may stretching   $U$   such
 that $ \mathbb{C}-\overline{B^{-1}\circ B(U)}  $  is connected. Therefore,

\emph{for each $z\in E$, there is an open neighborhood  $U$(
$U\subseteq E$)  of $z$ that admits a complete local inverse, and}
$ \mathbb{C}-\overline{B^{-1}\circ B(U)}  $  \emph{is connected.} %

Now we are ready to give the proof of Lemma \ref{64}.

\noindent\textbf{Proof of Lemma \ref{64}.} Since each operator in a von
Neumann algebra is the finite linear span of its unitary
operators\cite{Con}, we may assume that $S$ is a unitary operator in
$\mathcal{V^*}(B)$. Let $\widetilde{M_B}$ denote the multiplication
operator by $B$ on $L^2(\mathbb{D}, dA)$ given by
$$\widetilde{M_B}g=Bg, g\in L^2(\mathbb{D}, dA).$$
If we can show that $\widetilde{M_B}$ is the minimal normal
extension of $M_B$,
 then by  \cite[pp. 435, 436, Theorem VIII.2.20]{Con} and \cite[pp.552, 553]{DSZ}, for each unitary operator
$S\in V^*(B)$, there exists a unitary operator $\widetilde{S}$
defined on
 $L^2(\mathbb{D}, dA)$ such that $\widetilde{S}|_{L_a^2(\mathbb{D})}=S$ and $\widetilde{S}$ commutes with $\widetilde{M_B}$.
To see this,   set
$$K=span\{\widetilde{M_B}^{*k}h; \ h\in L_a^2(\mathbb{D}),k\geq 0\}.$$
Define a linear map $\widetilde{S}$  on $K$ such that
$\widetilde{S}(\widetilde{M_B}^{*k}h)= \widetilde{M_B}^{*k}Sh$.
Since $S$  is a unitary operator, it is easy to check that
$\widetilde{S}$ is well defined on $K$ and that $\widetilde{S}$ commutes
with $\widetilde{M_B}$. Moreover,
 $\widetilde{S}$ is an isometric operator whose range equals $K$. Thus, $\widetilde{S}$
can be immediately extended to a unitary operator on the closure
$\overline{K}$ of $K$. Since
$$K=span\{\widetilde{M_B}^{*k}\widetilde{M_B}^l h; \ h\in
L_a^2(\mathbb{D}),k\geq 0,l\geq 0\},$$ it follows that for any
polynomial $P$, $P(B,\overline{B})h \in  K$. Then applying
Weierstrass's Theorem shows that for any $u\in
C(\overline{\mathbb{D}})$, $u(B)h\in \overline{K}$. Lusin's Theorem
\cite[ p.242]{Hal} states that for each $v\in  L^\infty( \mathbb{D})
$, there is a uniformly bounded sequence $\{v_n\}$ in $
C(\overline{\mathbb{D}})$ such that $\{v_n\}$ converges
to $v$ in measure. Then it is easy to see that for any $u\in
L^\infty({\mathbb{D}})$, $u(B)h\in \overline{K}$. In particular, for
any open subset $U$ of $\mathbb{D}$,
$$\chi_{B^{-1}\circ B(U)}h\in  \overline{K}.$$ As mentioned above,
 for each $z\in E$  there is an open neighborhood  $U$ $(U\subseteq E$)  of $z$ that admits a complete local inverse, such that
  $$B^{-1}\circ B(U)=\bigsqcup_j \rho_j(U),$$ and
$ \mathbb{C}-\overline{B^{-1}\circ B(U)}  $   is connected. For each
fixed $i$, $\chi_{\rho_{i}(U)} $ denotes the characterization
function of $\rho_{i}(U)$.
As done in \cite{DSZ}, applying  the Runge theorem \cite{Con, Hor} shows that there is a sequence of polynomials $\{p_k\}$ such that 
$p_k$ converges  to  $\chi_{\rho_{i}(U)}  $ uniformly on the closure
of $B^{-1}\circ B(U)$. Therefore, since $\chi_{B^{-1}\circ
B(U)}p_k\in \overline{K}$, $\chi_{\rho_i(U)}\in \overline{K}.$ In
particular, $\chi_{U}\in \overline{K}.$
 By the same reasoning, for any   disk $ \Delta $ such that  $\Delta \subseteq U $, we   also
have $\chi_{ \Delta }\in \overline{K}$. Since all $\chi_{ \Delta }$
span a dense subspace of $L^2(U)$, then it follows that
$L^2(U)\subseteq \overline{K}$. Since all such open sets $U$ covers
$\mathbb{D}$, we have $L^2( \mathbb{D})\subseteq \overline{K}$, and
hence $L^2( \mathbb{D})= \overline{K}$. This shows that
$\widetilde{M_B}$ is exactly the minimal normal  extension  of
$M_B$, as desired.

Now we assume that the unitary operator $S\in \mathcal{V}^*(B)$ has
the following form:
 $$ Sh(z)=\sum_{ [\rho]}\sum_{\sigma \in [\rho]} c_\rho  h\circ \sigma(z) \sigma'(z),
 \ h\in L^2_a( \mathbb{D}),\, z\in E .$$
Let $U$ be  an open set that admits a complete local inverse, and $$B^{-1}\circ B(U)=\bigsqcup_j
\rho_j(U).$$ Since $\widetilde{S}$ commutes with $\widetilde{M_B}$,
by the theory of spectral decomposition, $\widetilde{S}$ also
commutes with  $M_F$, where $F=B^{-1}\circ B(\Delta)$ for any
measurable subsets $\Delta$ of $\mathbb{D}$\cite[pp.553-557]{DSZ}.
In particular, letting $\Delta=U$, we have
\begin{equation}\widetilde{S}M_{\chi_{B^{-1}\circ B(U)}}=M_{\chi_{B^{-1}\circ B(U)}}\widetilde{S}. \label{6.5}\end{equation}
By the above discussion,  there is a sequence of polynomials $\{p_k\}$ such that $p_k$
converges  to the characterization function $f\triangleq
\chi_{\rho_i(U)} $ uniformly on the closure of $B^{-1}\circ B(U)$.
Now by (\ref{6.5}), we have
\begin{eqnarray*}\widetilde{S} \chi_{B^{-1}\circ B(U)} p_k &=&  M_{\chi_{B^{-1}\circ B(U)}}Sp_k \\
&=& \sum_{ [\rho]}\sum_{\sigma \in [\rho]} \chi_{\sqcup \rho_j(U)}
c_\rho  p_k \circ \sigma(z) \sigma'(z)  .
\end{eqnarray*}
Letting $k\to\infty$ gives that
\begin{equation}\widetilde{S} f
 =  \sum_{ [\rho]}\sum_{\sigma \in [\rho]} c_\rho  f \circ \sigma(z) \sigma'(z)  .\label{6.6}\end{equation}
 The above discussions show that if $f=\chi_{\rho_i(V)}$ for some open subset $V$ of $U$, (\ref{6.6}) also holds.
 Thus, if $f$ is a linear span of finite $\chi_{\rho_i(V)}(V \subseteq U)$, (\ref{6.6}) holds. Therefore the 
identity (\ref{6.6}) can be generalized to the case when $f\in
C_c(B^{-1}\circ B(U)).$
 That is, locally, \emph{$\widetilde{S}$ has the same representation as $S$.}

By  (\ref{6.6}) and the proof of Lemma \ref{63},
$$(\widetilde{S})^* f
 =  \sum_{ [\rho]}\sum_{\sigma \in [\rho^-]} \overline{c_\rho}  f \circ \sigma(z) \sigma'(z) ,\  f\in C_c(B^{-1}\circ B(U)).$$
 Write $$S^*h(z)= \sum_{ [\rho]}\sum_{\sigma \in [\rho^-]} d_\rho  h \circ \sigma(z) \sigma'(z),\ h\in L^2_a( \mathbb{D}),\, z\in E ,$$
 and similarly,  for any  $f\in C_c(B^{-1}\circ B(U)) $ we have
 $$\widetilde{S^*}f=\sum_{ [\rho]}\sum_{\sigma \in [\rho^-]} d_\rho  f \circ \sigma(z) \sigma'(z) .$$

 Since $(\widetilde{S})^*=\widetilde{S^*}$, then by the uniqueness of the coefficients $c_\rho,$ we get $d_\rho =\overline{c_\rho}.$
 Therefore, $$S^*h(z)=\sum_{ [\rho]} \overline{c_\rho}   h \circ \sigma(z) \sigma'(z)\equiv
 \sum_{ [\rho]} \overline{c_\rho}  \mathcal{E}_{[\rho^-]} h(z),\ h\in L^2_a( \mathbb{D}),\, z\in E.$$
 This gives the former part of Lemma \ref{64}. The remaining part follows directly from Lemma 3.5.
 $\hfill \square$
 \begin{rem} By further effort, one can show that (\ref{6.6}) holds for any $f$ in $C_c(\mathbb{D})$, a dense subspace of $L^2(\mathbb{D})$.
 \end{rem}

It is natural to ask  that for which thin Blaschke product $B$,
$M_B$ has a nontrivial reducing subspace. The following provides such an example.
\begin{exam} Suppose that  a Blaschke product $ \phi$ has the form $\phi =B\circ \varphi$, where $B\in H^\infty (\mathbb{D})$ and  $\varphi $ is
a finite Blaschke product with $\deg \varphi \geq 2.$ Then we have
$\mathcal{V}^*(\phi ) \supseteq  \mathcal{V}^*(\varphi )$. Since $
\mathcal{V}^*(\varphi )$ is nontrivial\cite{HSXY}, then
$\mathcal{V}^*(\phi )$ is nontrivial.
 For example,  it is easy to check that  for each thin Blaschke product $B$ and $n\geq 2$,
$B(z^n)$ is always a  thin Blaschke product.

For more examples,  let $\gamma$ be a non-tagential $C^1$-smooth curve  such that $\gamma(1)\in  \mathbb{T}$ and $\gamma'(1)\neq 0$. Pick
a thin Blaschke sequence  $\{z_n\} $ lying on $\gamma$, and
 denoted by $B$ the  Blaschke product for $\{z_n\} $. Then we claim that for any finite Blaschke product $\psi$,
 $B\circ \psi $ is also a thin Blaschke product.

 The reasoning is as follows. It is known that the derivative of a finite
 Blaschke product never vanishes on $\mathbb{T}$.
 Then it follows that there is an $r_0\in [0,1)$ such that $\psi^{-1}(\gamma [r_0,1])$ consists of
 $d$  $(d=\deg \psi)$ disjoint arcs,
 $$\gamma_1,\cdots,\gamma_d.$$ Without loss of generality,
 we may assume that $r_0=0$, and then for each $z_k$, $\psi^{-1}(z_k)$ consists of $d$ different complex number, denoted by
 $$w_k^1,\cdots, w_k^d,$$
 where $w_k^j\in \gamma_j(1\leq j \leq d)$ for each $k.$
 To show that $B\circ \psi$ is a thin Blaschke product, it suffices to show that
 $$\lim_{k\to \infty} |(B\circ \psi)'(w_k^j)|(1-|w_k^j|^2)=1, j=1,\cdots, d.$$
 Since  $\lim\limits_{k\to \infty} | B'(z_k )|(1-|z_k |^2)=1$,
 it is reduced to show that $$\lim_{k\to \infty}  | \psi'(w_k^j)| \frac{1-|w_k^j|^2}{1-|z_k |^2}=1, j=1,\cdots, d.$$

 Now for a fixed $j$, rewrite $w_k$ for $w_k^j$. Write $z^*=\gamma(1) (z^*\in \mathbb{T} )$, and then
 $z_k$ tends to $ z^*$ as  $ k$ tends to  infinity.
 Also, there is some $w^* \in \mathbb{T}$ such that
$w_k$ tends to $ w^*$ as  $ k$ tends to  infinity. Precisely, this $w^*$ is exactly the endpoint of $\gamma_j$.
Since $\psi'(w^*)\neq 0$, $\psi$ is conformal at $w^*$. Then by the
conformal property, one may assume that there is a $\theta_0\in
(-\frac{\pi}{2},\frac{\pi}{2})$ such that
$$z_k=z^*+\varepsilon_k z^* e^{i\theta_0}+o(1) \varepsilon_k,k\to \infty,$$
and
$$w_k=w^*+ \delta_k w^* e^{i\theta_0}+o(1) \delta_k, k\to \infty .$$
Here, both $\varepsilon_k$ and $ \delta_k$ are positive
infinitesimals, and $o(1)$ denotes an infinitesimal. By easy computations, we have \begin{equation}
\frac{1-|z_k|^2}{1-|w_k|^2}\Big/  \frac{\varepsilon_k }{
\delta_k}\to 1, k\to \infty. \label{6.7} \end{equation} Since
$$\frac{\psi(w_k)-\psi(w^*)}{w_k-w^*}=\frac{z_k-z^* }{w_k-w^*}=
\frac{z^*}{w^*}\frac{1- z_k \overline{z^*} }{1-w_k\overline{w^*}},$$ it is easy to
verify  that $$\big|\frac{\psi(w_k)-\psi(w^*)}{w_k-w^*}\big| \Big/
\frac{\varepsilon_k }{ \delta_k}\to 1, k\to \infty.$$ Therefore,
$|\psi'(w^*)| /  \frac{\varepsilon_k }{ \delta_k} \to 1, k\to
\infty.$ Since $\psi'$ is continuous on $\overline{\mathbb{D}}$, we
have $|\psi'(w_k)|\to |\psi'(w^*)|,  k\to \infty,$ and hence
$$|\psi'(w_k)| \big/  \frac{\varepsilon_k }{ \delta_k} \to 1, k\to
\infty.$$ Then by (\ref{6.7}),
 $$\lim_{k\to \infty}   |\psi'(w_k )| \frac{1-|w_k |^2}{1-|z_k |^2}=1.$$
 The proof is complete. \label{68}
\end{exam}
For a Blaschke product $B$, if  there is a $\rho\in
\mathrm{Aut}(\mathbb{D})(\rho\neq id)$ satisfying $B\circ \rho =B$,
then \label{67} $$U_\rho:f\to f\circ \rho \, \rho', f\in L^2_a(\mathbb{D})$$ defines a unitary operator in $\mathcal{V^*}(B)$, and thus $\mathcal{V^*}(B)$
is nontrivial. However, under the setting of thin Blaschke products,
such a $B$ is strongly restricted. This is characterized by the
following, which is of interest itself.
\begin{prop} Given  a thin Blaschke product $B$, if  there is a function $\rho$ in $\mathrm{Aut}(\mathbb{D})(\rho\neq id)$ satisfying $B\circ \rho =B$,
then there exists a  thin Blaschke product $B_1$,   an integer
$n(n\geq 2)$ and a point  $\lambda\in \mathbb{D}$ such that \label{69}
$$B= B_1(\varphi_\lambda^n).$$   \end{prop}
\noindent  If $B$ is a finite Blaschke product, a similar version of
the above also holds, see \cite[Lemma 8.1]{DSZ}. Proposition
\ref{69} is sharp in the sense that it fails for interpolating
Blaschke products, especially for those holomorphic covering maps
$B:\mathbb{D}\to \mathbb{D}-F$, with $F$ a discrete subset of $
\mathbb{D}$. For details, refer to \cite[Examples 3.5 and 3.6]{GH2}.
\begin{proof} Assume that $B$ is a thin Blaschke product, and there is a $\rho$ in $\mathrm{Aut}(\mathbb{D})(\rho\neq id)$
satisfying $B\circ \rho =B$. Without loss of generality, we may assume that $B$ has only simple zeros; otherwise we shall replace
$B$ with $\varphi_\lambda\circ B $, and by Proposition \ref{24}(2)
$\varphi_\lambda\circ B$ is also a thin Blaschke product.

In the case of $B$ having only simple zeros, we first will
show that \emph{the group $G$ generated by $\rho$ is  finite.} To
see this, assume conversely that $G$ is not finite. 
Under this assumption, we first show that there is a
$z_0\in Z(B)$ such that for all   $i,j\in \mathbb{Z}$ with $i\neq
j$,
\begin{equation}\rho^i(z_0)\neq \rho^j(z_0). \label{6.8}\end{equation}
If (\ref{6.8}) were not true, then
 there would be a   positive integer $n_1$ satisfying
$$\rho^{n_1}(z_0)=z_0.$$
Now take some $z_1\in Z(B)$ such that $ z_1\neq z_0 $. If
$\{\rho^j(z_1)\}$ are pairwise different, then we are done;
otherwise, there would be a positive integer $n_2$ satisfying
$$\rho^{n_2}(z_1)=z_1.$$
Clearly,  $\rho^{n_1n_2}$  is a member in $\mathrm{Aut}(\mathbb{D})$
satisfying
$$\rho^{n_1n_2}(z_0)=z_0 \quad \mathrm{and} \quad \rho^{n_1n_2}(z_1)=z_1  .$$
That is, $\rho^{n_1n_2}$ has  two different fixed points  $z_0$ and
$z_1$ in $\mathbb{D}$. However, by \cite[Theorem 1.12]{Mi} each
member in $\mathrm{Aut}(\mathbb{D})$ admitting at least $2$
different fixed points must be the identity map, and hence
$\rho^{n_1n_2}=id$. This is a contradiction. Therefore, there exists
some $z_0\in Z(B)$ such that $\{\rho^j(z_0)\}$ are pairwise
different.

Thus, if $G$ were not infinite,
  then  $\{\rho^j(z_0)\}_{j\in \mathbb{Z}}$ were a
subsequence in $Z(B)$, and hence a thin Blaschke sequence. Thus we
would have $$\lim_{j\to  \infty} \prod_{k\in \mathbb{Z},k\neq
j}d(\rho^j(z_0),\rho^k(z_0))=1.$$ However, since the hyperbolic
metric $d$ is invariant under Mobius maps,
$$\prod_{k\in \mathbb{Z},k\neq j}d(\rho^j(z_0),\rho^k(z_0)) \equiv \prod_{k\in \mathbb{Z}-\{0\}}d( z_0 ,\rho^k(z_0))<1,\quad j\in \mathbb{Z} .$$
This is a contradiction. Therefore,   $G$   is  finite.

Let $n$ denote the order of $G$, and then   $\rho^n=id,$ and $$G=\{\rho^j:1\leq j\leq n \}.$$
By the proof of  \cite[Lemma 8.1]{DSZ}, there is a $\lambda \in \mathbb{D}$ and an $n$-th root $\xi$ of unity such that
  $$\rho(z)= \varphi_\lambda (\xi  \varphi_\lambda(z)).$$
   Write    $  \rho^* \triangleq \varphi_\lambda \circ \rho \circ  \varphi_\lambda$,  and then  $ \rho^*(z)=\xi z$.
 Notice that $B\circ \varphi_\lambda$ is also a thin Blaschke product satisfying
 $$(B\circ \varphi_\lambda)\circ \rho^*= B\circ \varphi_\lambda,$$ and $B=(B\circ \varphi_\lambda)\circ \varphi_\lambda$. To
 finish the proof, it suffices to discuss the case that $\rho$ is a rational rotation, i.e.
 $\rho (z)=\xi z$. In this case, we will show that $ B(z)=B_1(z^n)$ for some   thin Blaschke product $B_1$.

 For this, observe that
 $$B(\rho^j)=B,\,1\leq j <n,$$ which implies that for any $z_0  \in Z(B)$, $z_0$ and $\rho^j(z_0)$ are zeros of $B$, with the
 same multiplicity. Therefore, there is a subsequence $\{z_j\}$ in $Z(B)$   such that the zero sequence of $Z(B)$ is as follows:
 $$z_0,\rho(z_0),\cdots, \rho^{n-1}(z_0);z_1,\rho(z_1),\cdots, \rho^{n-1}(z_1);z_2,\cdots ,$$
 where $$\rho(z)=\xi z.Z$$
Notice that the zero sequence of $\varphi_{z_k^n}(z^n)$ is exactly
$z_k,\rho(z_k),\cdots, \rho^{n-1}(z_k)$. On the other hand,
$\{z_j\}$ satisfies the Blaschke condition:
$$\sum_{j=1}^\infty (1-|z_j|^2)<\infty.$$
 Since $ 1-|z_j^n|^2 =1-(|z_j|^2)^n\leq n (1-|z_j|^2),$  then $$\sum_{j=1}^\infty (1-|z_j^n|^2)<\infty.$$
 That is,  $\{z_j^n\}$  also gives a Blaschke product, denoted by $B_1$.
By omitting a unimodular constant, we have $B_1(z^n)=B(z)$. Since
$B$ is a thin Blaschke product, one sees that $B_1$ is also a thin
Blaschke product. The proof is complete.
\end{proof}
Therefore,  those $B$ in Proposition \ref{69} have already appeared
in Example \ref{67}.
 One may have a try to find some  thin Blaschke product $B$ of other type
such that $\mathcal{V}^*(B)$ is nontrivial. However, by Theorem
\ref{41} it seems that such a $B$  rarely appears.

\section{Thin   and finite Blaschke products}
~~~~ Very recently, it is   shown that $\mathcal{V}^*(\phi)$ are abelian for   finite Blaschke products $  \phi$\cite{DPW}.
  In this section, a similar version is presented for thin Blaschke
products. Also, we provide some examples of $\mathcal{V}^*(\phi)$, where  $ \phi$ are finite Blaschke products.

First, we give an example to show that there exists a thin Blaschke
product $B$ whose zero set $Z(B) $ satisfying $\overline{Z(B)}
\nsupseteq\mathbb{T}.$
\begin{exam}
 Let $\{c_n\}$ be a  sequence satisfying $c_n>0$ and $ \lim\limits_{n\to \infty}c_n=0.$ \label{71}
Suppose that $\{z_n\}$ is a sequence of points in the open unit disk
$\mathbb{D}$ such that
$$ \frac{1-|z_n|}{1-|z_{n-1}|}=c_n.$$
We will show that $\{z_n\}$ is a thin Blaschke sequence. To see
this,  given a positive integer $m$, let $k>m$ and consider the
product
\begin{equation} \prod_{j:j\neq k} d(z_k,z_j) \equiv \prod_{1\leq j\leq m} d(z_k,z_j) \prod_{j>m,j\neq k} d(z_k,z_j).
\label{7.1}\end{equation}

By the discussion on \cite[pp.203,204]{Hof}, if $\{w_n\}$ is  a
sequence  in   $\mathbb{D}$ satisfying
$$ \frac{1-|w_n|}{1-|w_{n-1}|}\leq c<1,$$
then $$ \prod_{j:j\neq k} d(w_k,w_j) \geq \big(\prod_{j=1}^\infty
\frac{1-c^j}{1+c^j}\big)^2. $$ Notice that the right hand side tends
to $1$ as $c$ tends to $ 0$. Now write \linebreak $d_m =\sup
\{c_j:j\geq m\}.$
 Since $ \lim\limits_{n\to \infty}c_n=0,$ then $ \lim\limits_{m\to \infty}d_m=0.$
For any $k>m$,  $$ \frac{1-|z_k|}{1-|z_{k-1}|}=c_k\leq d_m,$$ and
thus
$$  \prod_{j>m,j\neq k} d(z_k,z_j)\geq \big(\prod_{j=1}^\infty \frac{1-d_m^j}{1+d_m^j}\big)^2, \, k>m, $$
which implies that $$    \prod_{j>m,j\neq k} d(z_k,z_j) $$ tends to
$1$ as $m\to \infty$. For any $\varepsilon >0$, there is an $m_0$
such that
$$    \prod_{j>m_0,j\neq k} d(z_k,z_j)\geq 1- \varepsilon .$$
On the other hand,  $$\lim\limits_{k\to \infty } \prod_{1\leq j\leq
m_0} d(z_k,z_j)=1,$$ which, combined with (\ref{7.1}), implies that $$
\liminf_{k\to\infty} \prod_{j:j\neq k} d(z_k,z_j) \geq 1-\varepsilon
.$$ Therefore, by the arbitrariness of $\varepsilon$,
$$\lim_{k\to\infty}\prod_{j:j\neq k} d(z_k,z_j) =1;$$
that is, $\{z_k\} $ is a thin Blaschke product.

Now choose a dense sequence $\{r_n\}$ in $[0,1]$. If $\{z_n\}$ is
given as above, then $\{|z_n|\exp(2\pi i r_n)\}$ also gives a thin
Blaschke sequence. Moreover, we have  $\overline{\{|z_n|\exp(2\pi i
r_n)\}}\supseteq \mathbb{T}.$ For example,   set $|z_n|=1-
\frac{1}{n!}, $ and then $\{|z_n|\exp(2\pi i r_n)\}$ is a thin Blaschke sequence.
\end{exam}

It is of interest that the commutativity of
  the von Neumann algebras  $\mathcal{V}^*(B )$ for  thin   Blaschke products $B$ is connected with that of   finite Blaschke products, by the following.
\begin{prop}The followings are equivalent: \label{72}
\item[$(i)$] For any thin   Blaschke product $B$ with $\overline{Z(B)} \nsupseteq \mathbb{T}$, $\mathcal{V}^*(B )$ is abelian.
\item [$(ii)$]For any finite Blaschke product $  \phi$, $\mathcal{V}^*(\phi)$ is abelian.
\item [$(iii)$] For any finite Blaschke product $  \phi$,  $M_\phi$
 has exactly $q$ minimal reducing subspaces, where $q$ equals the number of components of $S_\phi$.
\end{prop}
\begin{proof} For (ii) $\Leftrightarrow $ (iii), see \cite{DSZ}.

To prove (i) $\Rightarrow$ (ii), consider the thin Blaschke product
$\psi$ whose zero sequence is $\{1-\frac{1}{n!}\}$, see Example
\ref{71}. Then for a given finite Blaschke product $  \phi$, write
$B\triangleq \psi\circ \phi $. By Example \ref{68}, $B$ is also a
thin Blaschke product and it is clear that $\mathcal{V}^*(B )
\supseteq \mathcal{V}^*(\phi )$. Since $\mathcal{V}^*(B )$ is
abelian, then so is $\mathcal{V}^*(\phi )$.

It remains to show that (ii) $\Rightarrow$ (i). The following
discussion depends heavily on the geometric structure of a thin
Blaschke product. Now suppose that  $B$ is a thin   Blaschke product
with $\overline{Z(B)} \nsupseteq \mathbb{T}$. Without loss of
generality, we may assume that $\mathcal{V}^*(B )$  is not trivial.
By Corollary \ref{65},  $\mathcal{V}^*(B )$ is generated by
$S_{[\sigma]}$, where $[\sigma]$ is
 a component of $S_B$ satisfying $\sharp [\sigma]<\infty$ and $\sharp [\sigma^-]<\infty.$ It suffices to show that
 all such operators $S_{[\sigma]}$ commutes.

 Now observe that for any  $[\sigma]$ and $[\tau]$,
  $S_{[\sigma]} S_{[\tau]}$ has the form $$\sum_{i}S_{[\sigma_i]},$$
  where the sum is finite and some $\sigma_i$ may lie in the same class; by Lemma
  \ref{64}   each $\sigma_i$ satisfying  $\sharp [\sigma_i]<\infty$ and $\sharp [\sigma_i^-]<\infty.$
   Based on this observation, we define \emph{the composition} $[\tau]\circ [\sigma]$ to be a formal sum
  $\sum_{i} [\sigma_i] $; and
  its graph $G([\tau]\circ [\sigma]) $ is defined to be the  \emph{disjoint} union of   $G[\sigma_i]$, denoted by $\sum_{i} G[\sigma_i] $.
  For example, if both  $\sigma$ and $\tau$ are in $\mathrm{Aut}(\mathbb{D})$, then $G([\tau]\circ [\sigma])=G([\tau\circ \sigma])$.

 Clearly,   $S_{[\sigma]} S_{[\tau]}= S_{[\tau]} S_{[\sigma]} $  if and only if $[\sigma]\circ [\tau]= [\tau] \circ [\sigma]$;
  if and only if $G([\sigma]\circ [\tau])= G([\tau] \circ [\sigma])$.
  Next, we will take a close look at $G([\sigma]\circ [\tau])$ and $ G([\tau] \circ [\sigma])$. Write $m= \sharp[\sigma]$ and  $n= \sharp[\tau]$.
  Fix a point $z_0\in E$, there are exactly $m$ points  in $G[\sigma]$ with the form $(z_0,\cdot)$:
  $$(z_0,\xi'_1),\cdots, (z_0,\xi'_m).$$
  For each $j=1,\cdots, m$, there are $n$ points in $G[\tau]$ with the form $(\xi'_j,\cdot)$:
  $$(\xi'_j,\xi_j^k), \ k=1,\cdots, n.$$
For each $j$ and $k$,   $(z_0, \xi_j^k)$ is in some $G[\sigma_i]$ where $[\sigma_i]$ appears in the sum of \linebreak $[\tau]\circ  [\sigma]$. %
By omitting the order of $(z_0, \xi_j^k)$, the unordered  sequence
$\{(z_0, \xi_j^k)\}_{j,k}$ represents $G([\tau]\circ  [\sigma])$.
 Similarly, for the same  $z_0\in E$, there are exactly $n$ points in $G[\tau]$:
   $$(z_0,\eta'_1),\cdots, (z_0,\eta'_n).$$
  For each $j=1,\cdots, n$, there are $m$ points with the form $(\eta'_j,\cdot)$ in $G[\sigma]$:
  $$(\eta'_j,{\eta}_j^k), \ k=1,\cdots, m.$$
  Also, the sequence $\{(z_0, \eta_j^k)\}_{j,k}$ represents $G([\sigma]\circ [\tau]) $.
Then $$G([\sigma]\circ [\tau])=G( [\tau] \circ [\sigma])$$ if and
only if the sequences $\{(z_0, \xi_j^k)\}_{j,k}$ and  $\{(z_0,
\eta_j^k)\}_{j,k}$ are equal, omitting the order.

 Below,  to finish the proof, we will use some geometric property of thin Blaschke product.
  As done in the proof of Corollary \ref{46}, let $\Omega_t$ be a simply-connected
domain whose the boundary   is an analytic   Jordan curve; moreover,
$\Omega_t$ is $*$-connected.
  Also, we may assume that $t$ is enough large such that $z_0$ and all finitely many points $\xi'_j,\xi_j^k;\eta'_j,\eta_j^k$
  lie in $\Omega_t$. Again by the proof of Corollary \ref{46},
 there is a biholomorphic map $f:\mathbb{D}\to  \Omega_t$ such that $B\circ f=t\phi$ for some finite Blaschke product.
 The equivalences $[\sigma]$ and $ [\tau]$ naturally grows on $\Omega_t$, and by the above paragraph, the compositions $[\tau]\circ  [\sigma]$
 and $[\sigma]\circ [\tau] $ also make  sense on  $\Omega_t$.
 By the identity $B\circ f=t\phi$, the map
 $$[\rho]\mapsto [\rho^*]\triangleq [f^{-1}\circ \rho \circ f]$$
  maps each component $[\rho]\in S_{B|_{\Omega_t}}$ to $[\rho^*]\in S_\phi$.
We will pay special attention to $\rho=\sigma$ or $ \tau$. Since
$\mathcal{V}^*(\phi)$ is abelian, then $S_{[\sigma^*]} S_{[\tau^*]}=
S_{[\tau^*]} S_{[\sigma^*]} $, which implies that $G([\sigma^*]\circ
[\tau^*])= G([\tau^*] \circ [\sigma^*])$. This immediately gives
that $G([\sigma]\circ [\tau])=G( [\tau] \circ [\sigma])$, and hence
$S_{[\sigma]} S_{[\tau]}= S_{[\tau]} S_{[\sigma]} $. Therefore,
$\mathcal{V}^*(B )$ is abelian. The proof is complete.
\end{proof}

Very recently, Douglas, Putinar and Wang\cite[Theorem 2.3]{DPW} showed that   $\mathcal{V}^*(\phi)$ is abelian for
any finite Blaschke product $\phi$. Thus,  we present our main result in this section as follows.
\begin{cor}For any thin   Blaschke product $B$ satisfying $\overline{Z(B)} \nsupseteq\mathbb{T}$, $\mathcal{V}^*(B )$ is abelian. \label{73}
\end{cor}
 Now turn to finite Blaschke products. The following shows that even if
 the assumption of  Corollary \ref{42} fails, $   \mathcal{V}^*(B) $ may be also $2$-dimensional. In fact, it is a consequence of
 Lemma \ref{51}.
\begin{prop}  Suppose that $B$ is a finite  Blaschke product. \label{74}
Then  there is an  $s(0<s<1)$ such that for any $ \lambda\in \mathbb{D}$
with $|\lambda|>s$, $\dim \mathcal{V}^*(B\varphi_\lambda)=2,$ and hence
$\mathcal{V}^*(B\varphi_\lambda)$ is abelian.
\end{prop}
\begin{proof} Suppose that $B$ is a finite  Blaschke product, and $s$ is a number in $(0,1)$
as in Lemma \ref{51}. For a fixed $\lambda\in \mathbb{D}$ satisfying
$|\lambda|>s$, write $$B_1=B\varphi_\lambda. $$ By the proof of Lemma \ref{51},
 there are two disjoint subdomains $\Omega_1$ and $\Omega_2$ of
$\mathbb{D}$, which are components of $\{z\in \mathbb{D}:|B_1(z)|<t
\}$ for some $t>0$. For some $w\in \Omega_2\cap E$, there are $n$
different points in $\Omega_1$, $w_1,\cdots, w_n$, such that
$B_1(w_j)=B_1(w),j=1,\cdots,n.$ Since $\Omega_1$ is $*$-connected
and $B_1 |_{ \Omega_2}$ is univalent, then by applying Lemma
\ref{48}, one can show that
$$G[(w,w_1)]=G[(w,w_j)], j=2,\cdots, n.$$ Thus, $G[(w,w_1)]$ and the
trivial component  $G[(w,w)]$ are the only components of $S_B$, and
hence $ \mathcal{V}^*(B_1)$ is generated by $S_{[(w,w_1)]}$ and $I$.
Therefore, $\dim \mathcal{V}^*(B_1)=2,$ i.e. $\dim
\mathcal{V}^*(B\varphi_\lambda)=2.$ Theorem III.1.2 in \cite{Da} states
that any finite dimensional von Neumann algebra is $*$-isomorphic to
the direct sum $\bigoplus_{k=1}^r M_{n_k}(\mathbb{C}) $ of full
matrix algebras. Therefore, $\mathcal{V}^*(B\varphi_\lambda)$ is abelian.
\end{proof}
The following provides another example  of von Neumann algebras  $
\mathcal{V}^*(B) $, where $B$ is a finite Blaschke product. We give
some estimates for $\dim \mathcal{V}^*(B) $.
\begin{exam} Let  $\lambda_1,\lambda_2, \lambda_3$ be three different  points in $
\mathbb{D}$ and  $\lambda_1\lambda_2\lambda_3\neq 0$.
 Write $B=z\varphi_{\lambda_1}^m(z)\varphi_{\lambda_2}^n(z)\varphi_{\lambda_3}^l (z)$.
 Then  $\dim
\mathcal{V}^*(B)\leq 4. $ The reasoning is as follows. Pick a point
$z_0$ that is enough close to $0$. By Theorem \ref{49}, it is not
difficult to see that there are $m+n+l+1$ points in
$B^{-1}(B(z_0))$:
 $$z_0; z_1^1,\cdots,z_1^m;z_2^1,\cdots,z_2^n;z_3^1,\cdots,z_3^l,$$
 where   $z_i^j$ are around $\lambda_i$$(i=1,2,3)$ for all possible $j$.
 Then by the proof of Lemma \ref{48}, for $ i=1,2,3$ we have $G[(z_0,z_i^1)]=G[(z_0,z_i^j)]$, $j=1,2,\cdots.$
 Thus, there are at most $4$ components of $S_B$: $$G[(z_0,z_0)]\equiv G[z],G[(z_0,z_1^1)],
 G[(z_0,z_2^1)] \  \mathrm{and} \  G[(z_0,z_3^1)].$$
Then by \cite[Theorem 7.6]{DSZ} or   Corollary \ref{65}, $\dim
\mathcal{V}^*(B)\leq 4. $

Similarly, if $B=z^n\varphi_{\lambda }(z) (\lambda \neq 0)$, then  $ \dim
\mathcal{V}^*(B)= 2$. In this case, $M_B$ has exactly $2$ minimal
reducing subspaces. In particular, if $n=3$, this is exactly
\cite[Theorem 3.1]{SZZ1}.
\end{exam}

%

 \vskip3mm \noindent{Kunyu Guo, School of Mathematical Sciences,
Fudan University, Shanghai, 200433, China, E-mail:
kyguo@fudan.edu.cn}

\noindent{Hansong Huang, Department of Mathematics, East China
University of Science and Technology, Shanghai, 200237, China,
E-mail: hshuang@ecust.edu.cn
\end{document}